\documentclass[a4paper,11pt]{article}
\usepackage{mathtools}
\usepackage{amsmath}
\usepackage{mathrsfs}
\usepackage{amssymb}
\usepackage{amsthm}
\usepackage{dsfont}
\usepackage{graphicx}
\usepackage{bmpsize}
\usepackage{times}
\usepackage{ textcomp }
\usepackage{mathtools}
\usepackage{latexsym, empheq, fancybox}
\usepackage{mathrsfs}
\usepackage{exscale}
\usepackage{booktabs, array}
\usepackage[english]{babel}
\newtheorem{theorem}{Theorem}[section]
\newtheorem{lemma}[theorem]{Lemma}
\newtheorem{definition}[theorem]{Definition}

\newtheorem{corollary}[theorem]{Corollary}

\usepackage{authblk}
\title{\Large\textbf{Optimal Actuator and Observation Location for Time-Varying Systems on a Finite-Time Horizon}
\thanks{This work was
supported by HITEC and IEK-8 of Forschungszentrum J\"ulich.}}

\author{Xueran Wu\thanks{IEK-8, Forschungszentrum J\"ulich, Wilhelm-Johnen-Stra\ss e,
52428, J\"ulich, Germany.
       {\tt\small Corresponding to x.wu@fz-juelich.de}} \ $^{\ddagger}$,\    
Birgit Jacob\thanks{Mathematics Department, Unversity of Wuppertal,
Gau\ss stra\ss e 20, 42119, Wuppertal, Germany.  {\tt\small xueranwu@uni-wuppertal.de, jacob@math.uni-wuppertal.de}.} 
\ \ and Hendrik Elbern$^{\dagger}$\thanks{Rhenish Institute for Environmental Research,  University of Cologne, Aachener Stra\ss e 209,
50931, Cologne, Germany. {\tt\small\quad he@eurad.uni-koeln.de}.
        }
}

\begin{document}
\date{}

\maketitle
\begin{abstract}
The choice of the location of controllers and observations is of great importance for designing
control systems and improving the estimations in various practical problems. For time-varying systems in Hilbert spaces,
the existence and convergence of the optimal location based on linear-quadratic control on a finite-time horizon is studied.
The optimal location of observations for improving the estimation of the state at the final time, based on Kalman filter, is considered as the dual problem to the LQ
optimal problem of the control locations. Further, the existence and convergence of optimal locations of observations for improving the estimation at the initial time,
based on Kalman smoother is discussed. The obtained results are applied to a linear advection-diffusion model.
\end{abstract}
\noindent
\textbf{Keywords:} Approximation, Kalman filter, Kalman smoother, linear-quadratic control, optimal observation location

\section{Introduction}
The choice of the locations of control hardware, such as sensors and actuators,
plays an important role in the designs of control systems for many physical and engineering problems.
Proper locations of sensors and actuators is essential to improve the performance of the controlled system.
Many researchers
have focused on the study of finding the optimal locations of control hardware and different criteria of
optimising control locations were established, such as maximization of observability and controllability
\cite{Johnson69}, \cite{Padula99}, minimizing the linear quadratic regulator cost \cite{Nam96}.
Geromel \cite{Geromel89} successfully reformulated the LQ cost function into a convex optimization problem by mapping the
locations of controller into zero-one vectors and expressed the solution of classic LQ problem in terms of a Riccati equation.
Morris \cite{Morris11} optimized controller locations of time-invariant systems on an infinite-time horizon in Hilbert spaces by solving an algebraic
Riccati equation and showed the convergence of optimal controller locations of a sequence of approximated finite-dimensional systems.
Further, an algorithm \cite{Darivandi13} for the linear quadratic optimal problem of controller locations based on the convexity shown in \cite{Geromel89} are introduced.

The issue of observations is also of great importance of many estimation problems for stochastic systems, such as weather forcasting and data assimilation in meteorology.
For this kind of problems, observations always have low temporal and spatial density. The lack of observations is a major barrier of preventing the improvement of estimations and leading to the accuracy of predictions. On one hand, the insufficient observations
become the main reason that many works are introduced
to improve approaches of estimations in in the recent years. On the other hand,
one possibility to improve the predictive or estimation skill for specific problems is to
target the locations of observations which can potentially result in the
largest forecast improvement in order to make observations more efficient. The better choice of locations of observations can help making more progress of the predictive or estimation skills. In contrast, improper observations probably make no sense to the accuracy of predictions and lead to the waste of resources by optimizing the improper parameters.
There are several papers focusing on this problem from the perspective of applications. For finite-dimensional systems in practice, approaches based on singular value decomposition (\cite{Buizza94}, \cite{Buizza99}) always help determining the direction with the strong influence of observations. However, it cannot solve the optimal problem of observation locations.
Motivated by problems of data assimilation in meteorology, we estimate unknown random variables by Kalman filter and smoother, which has been theoretical foundation of one of the most popular data assimilation approaches in last decade. In fact, since 1960's, besides of applications in meteorology, the Kalman filter and smoother \cite{Kalman61} were widely applied in many other fields to produce optimal linear estimations of states and parameters through a series of observations over time. It provides us an opportunity to define and search for optimal locations of observations by minimizing the covariance based on Kalman filter and smoother.

In this paper, we will start from the infinite-dimensional state space to consider the optimal location problem of controllers and observations for time-varying systems on a finite-time horizon.
First, we study the linear-quadratic optimal location control problem for both deterministic and stochastic systems and develop conditions guaranteeing the existence of optimal locations of linear quadratic control
problems in Section \ref{main results of lq}. Associated with practical applications, since optimal control problems cannot be solved directly in infinite-dimensional spaces,
a sequence of approximations of the original time-varying system have to be considered. Thus, in Section \ref{approximation of opt control}, analogical to the approximation theory of time-invariant systems,
we introduce the similar approximation conditions of evolution operators so as to ensure that the approximated control problems converge to the optimal control problem of the
original infinite-dimensional time-varying system. Further, we show the convergence of minimal costs and optimal locations of the sequence of approximations.
In Section \ref{kalman filter} and Section \ref{ks on hilbert},
we derive the Kalman filter and smoother of time-varying systems in the integral form on Hilbert spaces.
Then, by duality between Kalman filter and linear-quadratic optimal control, under certain conditions,
the nuclearity of the covariance can be guaranteed. In Section \ref{opt loc of kf and ks} based on Kalman filter and smoother,
the existence and convergence of optimal location observations of the estimation of the model state for stochastic systems is shown.
Finally, we apply the obtained results to a three-dimensional advection-diffusion model with the special construction of the emission rate in Section \ref{lq advection diffusion}.
In this example, the operator splitting technique with spatial and temporal discretization is applied to simulate the practical application in meteorology.

\section{Existence of optimal actuator locations}\label{main results of lq}

Throughout this paper, we will always assume that the state space of the time-varying system is a real separable Hilbert space $X$, and the input and output space are Hilbert spaces denoted by $U$ and $Y$, respectively. First, we introduce the notion of  mild evolution operators for the time-varying system.
\begin{definition}\label{mild evolution}
Denote $\Gamma_{a}^{b}: \{(t,s)\vert -\infty<a\leqslant s\leqslant t\leqslant b<\infty\}$.
We call $T(\cdot, \cdot): \Gamma_{a}^{b}\rightarrow \mathcal L(X)$  a {\em mild evolution operator} if
\begin{enumerate}
\item $ T(t,t)=I,$
\item $T(t,r)T(r,s)=T(t,s), \quad  a\leqslant s\leqslant r\leqslant t \leqslant b,$
\item $ T(\cdot, s): [s, b]\rightarrow \mathcal L(X)$ and $T(t, \cdot): [a, t]\rightarrow \mathcal L(X)$ are strongly continuous.
\item $\displaystyle \lambda:= \sup_{(t,s)\in \Gamma_{a}^{b}}\Vert T(t, s)\Vert<\infty.$
\end{enumerate}
\end{definition}

In the following we assume that $T(\cdot, \cdot): \Gamma_{a}^{b}\rightarrow \mathcal L(X)$ is a mild evolution operator, and $B\in L^\infty_{s}(a, b;U, X)$ with
$B^*\in L^\infty_{s}(a, b; X, U )$. Here
\begin{align*}
L^\infty_{s}(a, b; X, Y ):= \{F: [a,b]\rightarrow \mathcal L(X, Y)&\mid F \mbox{ is strongly measurable and}\\
& \Vert F\Vert_\infty:= \text{ess} \hspace{-1ex}\sup_{t\in[a,b]}\Vert F(t)\Vert<\infty\}.
\end{align*}
For an initial time $t_0\in[a,b]$,  we consider the time-varying system described by
\begin{equation}\label{system}
x(t)=T(t,t_0)x_0+\int_{t_0}^{t}T(t,s)B(s)u(s)ds, \quad  t\in [t_0,b],
\end{equation}
where $x_0\in X$ and $u\in L^2(t_0,b;U)$.
We are intersted in the following linear-quadratic optimal control problem.\\

{\parindent0cm\bf Linear-Quadratic Optimal Control Problem:}
Find for $x_0\in X$  a control $u_0\in  L^2(t_0,b;U)$ which minimizes the  cost functional
\begin{equation}\label{cost}
 J(t_0,x_0,u)=\langle x(b), Gx(b)\rangle+\int_{t_0}^{b}\| C(s)x(s)\|^2+\langle u(s), F(s)u(s)\rangle ds,
\end{equation}
where the function $x$ is given by \eqref{system}.
Here  $C\in L^\infty_{s}(a, b;X,Y)$, $G\in\mathcal L(X)$ and $F\in L^\infty_{s}(a, b;U, U)$,  $F(t)$ is  self-adjoint and nonnegative  for  fixed $t$, and  $F^{-1}\in L^\infty_{s}(a, b;U, U)$.\\

It is well known, see \cite{Gibson79}, that the linear-quadratic optimal control problem possesses for $x_0\in X$ a unique solution $u_0$, which is given by $u_0(t)=-L(t) x(t)$, $t\in [t_0,b]$, $L(t)=F^{-1}(t)B^*(t)\Pi(t)$, such that the minimum of the cost functional is given by
\begin{equation*}
\min_{u\in  L^2(t_0,b;U)}J(t_0, x_0, u)=J(t_0, x_0, u_0)=\langle x_0, \Pi(t_0)x_0\rangle,
\end{equation*}
where the self-adjoint nonnegative operator $\Pi(t)$ is the unique solution of the first integral Riccati equation (IRE)
\begin{align}\label{ire}
\Pi(t)x= &\, \,T^*(b,t)GT(b,t)x\nonumber\\
        &+\int_t^{b}T^*(s,t)[C^*(s)C(s)-\Pi(s)B(s)F^{-1}(s)B^*(s)\Pi(s)]T(s,t)xds
\end{align}
and the second IRE
\begin{align}\label{ire2}
\Pi(t)x=&\, T_{\Pi}^*(b,t)GT_{\Pi}(b,t)x\\
        &+\int_t^{b}T_{\Pi}^*(s,t)[C^*(s)C(s)+\Pi(s)B(s)F^{-1}(s)B^*(s)\Pi(s)]T_{\Pi}(s,t)xds,\nonumber
\end{align}
where
\begin{align*}
T_{\Pi}(t,\tau)x=T(t,\tau)x-\int_\tau^{t}T(t,s)B(s)F^{-1}(s)B^*(s)\Pi(s)T_{\Pi}(s,\tau)xds, \  (t,\tau)\in \Gamma_{a}^{b}.
\end{align*}

Now we consider the situation having the opportunity to choose $m$ locations to control and each location varies over a compact set $\Omega\subset \mathds R^l$. We indicate these $m$ locations by the parameter $r\in \Omega^m$, and denote the location-dependent input operator $B(\cdot)$ by $B_r(\cdot)$. Throughout the rest of the paper, by a time-varying system with location-dependent input operator the time-varying system \eqref{system} and the cost functional \eqref{cost} with $B_r$ instead of $B$ is meant. The corresponding solution of the IRE and the Riccati operator $L$ are denoted by $\Pi_r$ and $L_r$, respectively.

In most cases, the initial state $x_0$ is not fixed. This indicates several different ways to define the optimal actuator location problem. We take two possible ways into account here. The first one is to minimize the cost with the worst choice of initial value, which is
\begin{equation*}
\max_{\Vert x_0\Vert=1}\min_{u\in L^2(t_0,b;U)}J_r(x_0,u)=\max_{\Vert x_0\Vert=1}\langle x_0,\Pi_r(t_0)x_0\rangle=\Vert \Pi_r(t_0)\Vert.
\end{equation*}
Let $\ell^r(t_0):=\Vert \Pi_r(t_0)\Vert$, the optimal performance of $r$ is
$
\hat{\ell}(t_0)=\inf_{r\in \Omega^m}\Vert\Pi_r(t_0)\Vert.
$

The second one is to assume that the system is stochastic. Thus, we need to consider the trace of $\Pi_r(t_0)$ instead, since the trace indicates the sum of the diviation of the state vector in each coordinate.
Thus the evaluation of the particular performance of $r$ is given by the nuclear norm of $\Pi_r(t_0)$, which is $\ell_1^r(t_0)=\Vert \Pi_r(t_0)\Vert_1$. Further, the optimal performance is
$$
\hat{\ell}_1(t_0)=\inf_{r\in \Omega^m}\Vert\Pi_r(t_0)\Vert_1.
$$

For time-invariant problems on an infinite time horizon this problem was studied in \cite{Morris11}.
In this section we prove the existence of optimal control locations for deterministic as well as stochastic time varying systems on a finite-time horizon.
\begin{theorem}\label{ex operator}
Let $\{B_r\}_r$, $r\in \Omega^m$, be a family of compact operator valued functions with the property that $\lim_{r\rightarrow r_0}\Vert B_r-B_{r_0}\Vert_\infty=0$ for some $r_0\in\Omega^m$. 
Then the solutions of the corresponding integral Riccati equations $\Pi_r$ satisfy
\begin{equation*}
\lim_{r\rightarrow r_0}\Vert \Pi_r(t)-\Pi_{r_0}(t)\Vert=0,  \quad  t\in[a,b],
\end{equation*}
and there exists an optimal location $\hat{r}$ such that for any initial time $t_0\in[a,b]$,
\begin{equation*}
\hat{\ell}(t_0)=\Vert\Pi_{\hat{r}}(t_0)\Vert=\inf_{r\in \Omega^m}\Vert\Pi_r(t_0)\Vert.
\end{equation*}
\end{theorem}
\begin{proof}
Thanks to the assumptions on $B_r$, there exists $\delta>0$ such that
$\lambda_B:= \sup \{ \Vert B_r(t)\| \mid t\in[a, b], \Vert r-r_0\Vert\leqslant\delta\}<\infty$. We denote by
$\mathbf B(r_0, \delta)$ the set $\mathbf B(r_0, \delta):=\{r\in \Omega^m: \Vert r-r_0\Vert\leqslant \delta\}$.
Thus, \cite[Theorem 5.1]{Gibson79} implies for every $x\in X$
\begin{equation*}
\Pi_r(t)x\rightarrow \Pi_{r_0}(t)x,\quad   r\rightarrow r_0.
\end{equation*}
For any feedback control $\tilde u(t)=\tilde{L}(t)x(t)$, $\tilde{L}\in L^\infty_s(a, b;X,U)$,
\begin{eqnarray*}\label{Pi bound}
&&\langle x(t), \Pi_r(t)x(t)\rangle\leqslant J(t,x(t),\tilde u)\\\nonumber
&=&\langle x(b), Gx(b)\rangle+\int_{t}^{b} \|C(s)x(s)\|^2+\langle \tilde{L}(s)x(s), F(s)\tilde{L}(s)x(s)\rangle ds \\\nonumber &=&\Vert G^{\frac{1}{2}}T_{\tilde{L},r}(b,t)x(t)\Vert^2\\
&&+\int_{t}^{b}\Vert C(s)T_{\tilde{L},r}(s,t)x(t)\Vert^2+\Vert F^{\frac{1}{2}}(s)\tilde{L}(s)T_{\tilde{L},r}(s,t)x(t) \Vert^2 ds,
\end{eqnarray*}
where $T_{\tilde{L},r}(t,\tau)x=T(t,\tau)x+\int_\tau^{t}T(t,s)B_r(s)\tilde{L}(s)T_{\tilde{L},r}(s,\tau)xds$, $(t,\tau)\in \Gamma_{a}^{b}$.

Since the family $B_r$ is uniformly bounded by $\lambda_B$ on $\mathbf B(r_0, \delta)$, \cite[Theorem 2.1]{Gibson79} implies for all $ r\in \mathbf B(r_0, \delta), (t,\tau)\in \Gamma_{a}^{b}$, $\|T_L(t,\tau)\|\le
\lambda \exp({\lambda \lambda_B\Vert \tilde L\Vert_\infty(t-\tau)}).$
Further, because $C\in L^\infty_s(a, b;X, Y)$, $F\in L^\infty_s(a, b;U, U)$, there exists a constant $\lambda_\Pi$, independent of $t$ and $r\in \mathbf B(r_0, \delta)$, such that
$
\Vert\Pi_r\Vert_\infty \leqslant \lambda_\Pi.
$

For $S_r=C^*C-L_r^*FL_r$, where $L_r=F^{-1}B_r^*\Pi_r$,  we obtain
\begin{equation*}
\Pi_r(t)x-\Pi_{r_0}(t)x=\int_t^{b}T^*(s,t)\left(S_r(s)-S_{r_0}(s)\right)T(s,t)xds, \quad  x\in X.
\end{equation*}
Since  $F^{-1}\in L^\infty_s(a, b;U,U)$ and the operator $B_{r_0}(t)$ is compact
for any $t\in [a,b]$, we have
\begin{eqnarray*}
\Vert L_r^*(t)-L_{r_0}^*(t)\Vert&\leqslant&\Vert F^{-1}\Vert_\infty(\Vert \Pi_r(t)\Vert\Vert B_r(t)-B_{r_0}(t)\Vert\\
&&+\Vert (\Pi_r(t)-\Pi_{r_0}(t))B_{r_0}(t)\Vert)\longrightarrow 0, \quad  r\rightarrow r_0,
\end{eqnarray*}
which shows
\begin{eqnarray*}
\Vert S_r(t)-S_{r_0}(t)\Vert
                 &\leqslant& \Vert L_{r_0}^*(t)-L_r^*(t)\Vert\Vert F(t)L_{r_0}(t)\Vert\\
&&+\Vert L_r^*(t)F(t)\Vert\Vert L_{r_0}(t)-L_r(t)\Vert\longrightarrow 0, \quad  r\rightarrow r_0.
\end{eqnarray*}
From the uniform boundedness of $F$, $B_r$ and $\Pi_r$ on $ \mathbf B(r_0, \delta)$, $L_r$ and further $S_r$ are uniformly bounded for all $t\in[a,b]$ and
$\mathbf B(r_0, \delta)$. According, thanks to the dominated convergence theorem, we obtain
$\Vert \Pi_r(t)-\Pi_{r_0}(t)\Vert\rightarrow 0, \ r\rightarrow r_0.$

Additionally, since $r\in \Omega^m$, $\Omega^m$ is a compact set, 
there exists an optimal location $\hat{r}$ such that $\Vert\Pi_{\hat{r}}(t_0)\Vert=\inf_{r\in \Omega^m}\Vert\Pi_r(t_0)\Vert.$
\end{proof}

Theorem \ref{ex operator} shows the continuity of optimal actuator locations and existence of the optimal location in the operator norm. For stochastic systems, the above problem leads to the nuclear norm.
Thus, first we develop conditions which guarantee that the Riccati operator is a nuclear operator. Similar to \cite[Theorem 3.1]{Curtain01},  we have
\begin{theorem}\label{nuclear}
Let $T(\cdot,\cdot)$ be a mild evolution operator on $X$, $B\in L^\infty_{s}(a, b;\mathbb C^p, X)$, and $C\in L^\infty_{s}(a, b;X, \mathbb C^q)$.
Then for any $t_0\in[a,b]$ we have:
\begin{enumerate}
\item The observability operator $\mathcal C_{t_0}: X\rightarrow L^2(t_0,b;\mathbb C^q)$ defined by
\begin{equation*}
(\mathcal C_{t_0}x_0)(\cdot)=C(\cdot)T(\cdot,t_0)x_0, \quad  x_0\in X,
\end{equation*}
is a Hilbert-Schmidt operator;
\item The controllability operator $\mathcal B_{t_0}: L^2(t_0, b; \mathbb C^p)\rightarrow X$ defined by
\begin{equation*}
 \mathcal B_{t_0}u=\int_{t_0}^{b}T(b,s)B(s)u(s)ds
\end{equation*}
is a Hilbert-Schmidt operator;
\item $\mathcal C_{t_0}^*\mathcal C_{t_0}$ and $\mathcal B_{t_0}\mathcal B_{t_0}^*$ are nuclear operators.
\end{enumerate}
\end{theorem}

\begin{proof}
1. Defining $\mathcal C_{t_0,i}: X\rightarrow L^2(t_0,b)$, $i\in \{1, \dots, q\}$
\begin{equation*}
(\mathcal C_{t_0,i}x_0)(s)=\langle C(s)T(s,t_0)x_0, e_i\rangle,  \quad  s\geqslant t_0,
\end{equation*}
where $\{e_i\}$ is the standard orthogonal basis of $\mathbb C^q$. We have
\begin{eqnarray*}
\vert(\mathcal C_{t_0,i}x_0)(s)\vert&=&\vert\langle C(s)T(s,t_0)x_0, e_i\rangle\vert \leqslant \Vert C(s)T(s,t_0)x_0\Vert \Vert e_i\Vert\\
                                   &\leqslant& \Vert C(s)\Vert\Vert T(s,t_0)\Vert\Vert x_0\Vert <\infty.
\end{eqnarray*}
\cite[Theorem 6.12]{Weidmann80} implies that $\mathcal C_{t_0,i}$ is Hilbert-Schmidt, that is,  for any orthogonal basis $\{\bar{e}_i\}$ of $X$, we have
$
\sum_{i=1}^{q}\sum_{j=1}^{\infty}\Vert\mathcal C_{t_0,i}\bar{e}_j\Vert_{L^2(t_0,b)}^2<\infty.
$
Since $$\Vert \mathcal C_{t_0} \bar{e}_j\Vert_{L^2(t_0,b)}^2=\sum_{i=1}^{q}\Vert\mathcal C_{t_0,i}\bar{e}_j\Vert_{L^2(t_0,b)}^2,$$ we have
\begin{equation*}
\sum_{j=1}^{\infty}\Vert \mathcal C_{t_0} \bar{e}_j\Vert_{L^2(t_0,b; \mathbb C^q)}^2 =\sum_{i=1}^{q}\sum_{j=1}^{\infty}\Vert\mathcal C_{t_0,i}\bar{e}_j\Vert_{L^2(t_0,b)}^2<\infty,
\end{equation*}
which shows that $\mathcal C_{t_0}$ is a Hilbert-Schmidt operator.

2. According to \cite[Theorem 6.9]{Weidmann80},  $\mathcal B_{t_0}$ is Hilbert-Schmidt if and only if $\mathcal B^*_{t_0}$ is Hilbert-Schmidt. An easy calculation shows $\mathcal B_{t_0}^*: X \rightarrow L^2(t_0, b; U)$,
\begin{equation*}
(\mathcal B_{t_0}^*x)(\cdot)= B_{t_0}^*(\cdot)T^*(b,\cdot)x
\end{equation*}
From part 1,  $\mathcal B_{t_0}^*$ is Hilbert-Schmidt, and so is $\mathcal B_{t_0}$.

3. Since $\Vert\mathcal C_{t_0}^*\mathcal C_{t_0}\Vert_{1}\leqslant \Vert\mathcal C_{t_0}^*\Vert_{HS}\Vert\mathcal C_{t_0}\Vert_{HS}<\infty$
and $\Vert\mathcal B_{t_0}^*\mathcal B_{t_0}\Vert_{1}\leqslant \Vert\mathcal B_{t_0}^*\Vert_{HS}\Vert\mathcal B_{t_0}\Vert_{HS}<\infty$,
$\mathcal C_{t_0}^*\mathcal C_{t_0}$ and $\mathcal B_{t_0}\mathcal B_{t_0}^*$ are nuclear operator.
\end{proof}

\begin{corollary}\label{coro}
Assume that the input space $U$ and the output space $Y$ are finite-dimensional and $G$ that is a nuclear operator, then the unique nonnegative self-adjoint solution $\Pi(t_0)$ of the integral Riccati equation
is a nuclear operator.
\end{corollary}
\begin{proof}
Defining the bounded operator $\mathcal C_{t_0}: X\rightarrow L^2( t_0,b; U\times Y)$ by
\begin{equation*}
(\mathcal C_{t_0}x_0)(\cdot) =\left(\begin{array}{cc}
                    C(\cdot)\\
                    F^{\frac{1}{2}}(\cdot)L(\cdot)
                  \end{array}\right)T_L(\cdot, t_0)x_0, \quad  L=F^{-1}B^*\Pi.
\end{equation*}
 $\mathcal C_{t_0}$ is Hilbert-Schmidt by Theorem \ref{nuclear}.1
The second IRE \eqref{ire2} can be rewritten as
\begin{equation*}
\Pi(t_0)x= T_L^*(b,t_0)GT_L(b,t_0)x+\mathcal C_{t_0}^*\mathcal C_{t_0}x, \quad  x\in X.
\end{equation*}
Form Theorem \ref{nuclear}.3 and the nuclearity of $G$, $\Pi(t)$ is a nuclear operator.
\end{proof}

\begin{lemma}\label{per lemma}
Assume  $T(\cdot,\cdot)$ and $T_i(\cdot,\cdot)$, $i\in \mathbb N$,  are  mild evolution operators which are  uniformly
bounded by $\lambda_T$, $D_i, D \in L^\infty_s(a, b;X,X)$ satisfy $\Vert D_i(t)x - D(t)x\Vert\rightarrow 0$ as $i\rightarrow \infty$
for every $x\in X$ and $\sup_{i}\{\Vert D_i\Vert_\infty,\Vert D\Vert_\infty\}\leqslant \lambda_D$. $T_{D_i}(\cdot, \cdot)$,
$T_D(\cdot, \cdot)$ denote the perturbed evolution operators corresponding to the perturbation of $T_i(\cdot, \cdot)$ by $D_i$ and $T(\cdot,\cdot)$ by D.
If $\Vert T_i(t,\tau)x-T(t,\tau)x\Vert\rightarrow 0$ as $ i\rightarrow \infty $ for $x\in X$, then for any $(t,\tau)\in \Gamma_{a}^{b}$ and $x\in X$,
\begin{equation*}
\Vert T_{D_i}(t,\tau)x-T_D(t,\tau)x\Vert\rightarrow 0,\quad  i\rightarrow \infty.
\end{equation*}
\end{lemma}
\begin{proof}
As in \cite{Curtain78}, we construct $T_{D_i}(t, \tau)$ as
$T_{D_i}(t, \tau)=\sum_{n=0}^\infty T_{D_i,n}(t, \tau),$
where
\begin{equation*}
T_{D_i,0}(t, \tau)=T_i(t, \tau), \quad  T_{D_i,n}(t, \tau)x=\int_\tau^tT_i(t,s)D_i(s)T_{D_i,n-1}(s,\tau)xds, \ x\in X.
\end{equation*}

By induction we obtain
$
\Vert T_{D_i,n}(t,\tau) \Vert\leqslant \lambda_T(\lambda_T\lambda_D)^n\frac{(t-\tau)^n}{n!}.
$
$T_D(t,\tau)$ can be constructed in a similar manner with the same upper bound.

Defining $d_{i,n}(t,\tau)=T_{D_i,n}(t,\tau)-T_{D,n}(t,\tau)$, we have
$d_{i,0}(t,\tau)=T_i(t,\tau)-T(t,\tau)$,
\begin{align*}
d_{i,n}(t,\tau)= & \int_\tau^tT_i(t,s)D_{i}(s)d_{i,n-1}(s,\tau)ds+\int_\tau^tT_i(t,s)[D_i(s)-D(s)]T_{D,n-1}(s,\tau)ds\\
&+ \int_\tau^t[T_i(t,s)-T(t,s)]D(s)T_{D,n-1}(s,\tau)ds.
\end{align*}
The uniform boundedness of $\{T_{D_i}(t,\tau)\}_{i\in\mathds N}$ and $T_D(t,\tau)$ implies
\begin{equation*}
\Vert\sum_{n=0}^\infty d_{i,n}(t,\tau)\Vert \leqslant \Vert T_{D_i}(t,\tau)\Vert+\Vert T_D(t,\tau)\Vert<\infty.
\end{equation*}
Due to $\sup_{i}\{\Vert D_i\Vert_\infty,\Vert D\Vert_\infty\}\leqslant \lambda_D$ and $T(\cdot,\cdot)$, $T_i(\cdot,\cdot)$ are uniformly bounded, the mild evolution operators $T_{D}(\cdot,\cdot)$, $T_{D_i}(\cdot,\cdot)$,  are uniformly bounded, and further for any $n\in \mathds N$,
$
\sup_{i}\sup_{(t,\tau)\in \Gamma_{t_0}^{b}}\Vert d_{i,n}(t,\tau)\Vert<\infty.
$
Meanwhile, since $\Vert D_i(t)x - D(t)x\Vert\rightarrow 0$,
$
\Vert d_{i,0}(t,\tau)x\Vert =\Vert T_i(t,\tau)x-T(t,\tau)x\Vert\rightarrow 0, \quad  i\rightarrow \infty.
$
Hence,
\begin{eqnarray}\label{diff}
\Vert d_{i,n}(t,\tau)x\Vert&\leqslant& \int_\tau^t\Vert T_i(t,s)\Vert \Vert D_{i}(s)\Vert \Vert d_{i,n-1}(t,\tau)x\Vert ds\nonumber\\
&&+\int_\tau^t\Vert T_i(t,s)\Vert\Vert [D_i(s)-D(s)]T_{D,n-1}(s,\tau)x\Vert ds \nonumber\\
&&+\int_\tau^t\Vert(T_i(t,s)-T(t,s))(s)T_{D, n-1}(s,\tau)x\Vert ds \longrightarrow 0, \quad  i\rightarrow \infty.
\end{eqnarray}
By dominated convergence theorem,
\begin{equation*}
\Vert T_{D_i}(t,\tau)x-T_D(t,\tau)x\Vert\leqslant \sum_{n=0}^\infty \Vert d_{i,n}(t,\tau)x\Vert\rightarrow 0, \quad  i\rightarrow \infty.
\end{equation*}
\end{proof}

\begin{corollary}\label{per corollary}
For any mild evolution operator $T(\cdot,\cdot)$ with uniform bound $\lambda_T$ and $D_i, D \in L^\infty_s(t_0, b;X,X)$ with $\sup_{i}\{\Vert D_i\Vert_\infty,\Vert D\Vert_\infty\}\leqslant \lambda_D$, if $\Vert D_i(t) - D(t)\Vert\rightarrow 0$, then for  $T_D(\cdot, \cdot)$ which is the perturbation of $T(\cdot, \cdot)$ by $D_i$ and $T_{D_i}(\cdot, \cdot)$ which is the perturbation evolution operator of  $T(\cdot,\cdot)$ by $D$, we have
\begin{equation*}
\Vert T_{D_i}(t,\tau)-T_D(t,\tau)\Vert\rightarrow 0,\quad   i\rightarrow \infty.
\end{equation*}
\end{corollary}
\begin{proof}
From the assumptions, let $T_i=T$ in Lemma \ref{per lemma}, replace \eqref{diff} by
\begin{eqnarray*}
\Vert d_{i,n}(t,\tau)\Vert &\leqslant& \int_\tau^t\Vert T(t,s)\Vert \Vert D_{i}(s)\Vert \Vert d_{i,n-1}(t,\tau)\Vert ds\\
&&+\int_\tau^t\Vert T(t,s)\Vert\Vert D_i(s)-D(s)\Vert\Vert T_{D,n}(s,\tau)\Vert ds \longrightarrow 0, \quad  i\rightarrow \infty.
\end{eqnarray*}

Then, we can prove the uniform convergence of $T_{D_i}(t,\tau)$ by the dominated convergence theorem in the similar way with Lemma \ref{per lemma}.
\end{proof}

\begin{theorem}\label{ex nuclear}
We consider the time-varying system \eqref{system} with the location-dependent input operators and the cost functional \eqref{cost}. Assume $\{B_r\}_{r\in\Omega^m}$ 
satisfies $\lim_{r\rightarrow r_0}\Vert B_r-B_{r_0}\Vert_\infty=0$, for some $r_0\in \Omega^m$,  $U$ and $Y$ are finite-dimensional and $G$ is a nuclear operator, then
\begin{equation*}
\lim_{r\rightarrow r_0}\Vert \Pi_r(t)-\Pi_{r_0}(t)\Vert_1=0,  \quad  t\in[t_0,b]
\end{equation*}
and there exists an optimal location $\hat{r}$ such that
\begin{equation*}
\hat{\ell}_1(t)=\Vert\Pi_{\hat{r}}(t_0)\Vert_1=\inf_{r\in \Omega^m}\Vert\Pi_r(t_0)\Vert_1.
\end{equation*}
\end{theorem}

\begin{proof}
Similar to Theorem \ref{ex operator}, there exists $\delta >0$ such that
$
\sup_{r\in \mathbf B(r_0, \delta)}\Vert B_r\Vert<\infty,  r_0\in \Omega^m
$
and for every $x\in X$ and $t\in [t_0,b]$,
\begin{equation*}
\Pi_r(t)x\rightarrow \Pi_{r_0}(t)x,\quad r\rightarrow r_0.
\end{equation*}

Further, from \eqref{Pi bound}, we have $\Pi_r$ are uniformly bounded with $\lambda_\Pi$ for any $t\in[t_0,b]$ and $r\in \mathbf B(r_0, \delta)$.

Defining the operator $\mathcal C_{t,r}: X\rightarrow L^2(t,b; U\times Y)$, $t\in[t_0,b]$,
\begin{equation}\label{definition C}
(\mathcal C_{t,r}x(t))(\cdot) =\left(\begin{array}{cc}
                    C(\cdot)\\
                    -F^{\frac{1}{2}}(\cdot)B_r^*(\cdot)\Pi_r(\cdot)
                  \end{array}\right)T_{L,r}(\cdot, t)x(t).
\end{equation}
Corollary \ref{coro} has shown $\mathcal C_{t,r}$ is a Hilbert-Schmidt operator and
\begin{equation*}
\Pi_r(t)= T_{L,r}^*(b,t)GT_{L,r}(b,t)+\mathcal C_{t,r}^*\mathcal C_{t,r}.
\end{equation*} is nuclear if $G$ is nuclear.

Now let us show that $\mathcal C_{t,r}$ uniformly converges to $\mathcal C_{t,r_0}$ in Hilbert-Schmidt norm. Let $\{e_i\}_{i=1}^{p+q}$ and $\{\bar{e}_i\}_{i=1}^{\infty}$ be respectively the orthogonal basis of $U\times Y$ and $X$, then
\allowdisplaybreaks
\begin{eqnarray*}
&&\Vert\mathcal C_{t,r}-\mathcal C_{t,r_0}\Vert_{HS}\\
&=&\sum_{i=1}^{\infty}\int_{t}^{b}\sum_{j=1}^{p+q}\langle(\mathcal C_{t,r}\bar{e}_i)(s)-(\mathcal C_{t,r_0}\bar{e}_i)(s),e_j\rangle_{U\times Y} ds\\
&=&\sum_{i=1}^{\infty}\int_{t}^{b}\sum_{j=1}^{p+q}\vert\langle\bar{e}_i,
                        T_{L,r}^*(s, t)[C^*(s),L_r^*(s)F^{\frac{1}{2}}(s)]e_j\\
                        &&-T_{L,r_0}^*(s, t)[C^*(s),L_{r_0}^*(s)F^{\frac{1}{2}}(s)]e_j
                   \rangle_{X}\vert^2ds\\
&=&\int_{t}^{b}\sum_{i=1}^{\infty}\sum_{j=1}^{p+q}\vert\langle\bar{e}_i,
                       T_{L,r}^*(s, t)[C^*(s), L_r^*(s)F^{\frac{1}{2}}(s)]e_j\\&&-T_{L,r_0}^*(s, t)[C^*(s), L_{r_0}^*(s)F^{\frac{1}{2}}(s)]e_j
                   \rangle_{X}\vert^2ds\\
&=&\sum_{j=1}^{p+q}\int_{t}^{b}\Vert T_{L,r}^*(s, t)[C^*(s), L_r^*(s)F^{\frac{1}{2}}(s)]e_j-T_{L,r_0}^*(s, t)[C^*(s), L_{r_0}^*(s)F^{\frac{1}{2}}(s)]e_j\Vert_X^2ds,
\end{eqnarray*}
where $L_r=F^{-1}B_r^*\Pi_r$.

From Theorem \ref{ex operator}, we have
$
\lim_{ r\rightarrow r_0}\Vert L_r(t)-L_{r_0}(t)\Vert= 0
$
and $\Vert L_r\Vert_\infty<\infty$.
Then,
$$
\lim_{ r\rightarrow r_0}\Vert B_r(t)L_r(t)-B_{r_0}(t)L_{r_0}(t)\Vert=0
$$
and
$
\Vert B_rL_r\Vert_\infty <\infty.
$
Hence, from Corollary \ref{per corollary}, for any $(s,t)\in \Gamma_{t_0}^{b}$,
$T_{L,r}(s, t)$ uniformly converges to $T_{L, r_0}(s, t)$. Therefore,
\begin{eqnarray}\label{inequality}
&&\Vert T_{L,r}^*(s, t)[C^*(s), L_r^*(s)F^{\frac{1}{2}}(s)]e_j-T_{L,r_0}^*(s, t)[C^*(s), L_{r_0}^*(s)F^{\frac{1}{2}}(s)]e_j\Vert_X\nonumber\\
&\leqslant&\Vert (T_{L,r}^*(s, t)-T_{L,r_0}(s, t)^*)[C^*(s),L_r^*(s)F^{\frac{1}{2}}(s)]e_j]\Vert\nonumber\\
&&+\Vert T_{L,r_0}^*(s, t)[0,(L_r^*(s)-L_{r_0}^*(s))F^{\frac{1}{2}}(s)]e_j]\Vert\longrightarrow0, \quad  r\rightarrow r_0.
\end{eqnarray}

By dominated convergence theorem,
$
\Vert\mathcal C_{t,r}-\mathcal C_{t,r_0}\Vert_{HS}\rightarrow 0, \quad r\rightarrow r_0.
$
Further, if $G$ is a nuclear operator,
\begin{align*}
&\Vert\Pi_r(t)-\Pi_{r_0}(t)\Vert_1\\&\leqslant\Vert T_{L,r}^*(b,t)-T_{L,r_0}^*(b,t)\Vert \Vert GT_{L,r}(b,t)\Vert_1+\Vert T_{L,r_0}^*(b,t)G\Vert_1\Vert T_{L,r}(b,t)-T_{L,r_0}(b,t)\Vert\\
&+\Vert\mathcal C_{t,r}^*-C_{t,r_0}^*\Vert_{HS}\Vert\mathcal C_{t,r}\Vert_{HS}+\Vert C_{t,r_0}^*\Vert_{HS}\Vert\mathcal C_{t,r}-C_{t,r_0}\Vert_{HS}\rightarrow 0, \ r\rightarrow r_0.
\end{align*}

By the compactness of $\Omega^m$, the optimal location $\hat{r}$ exists in nuclear norm.
\end{proof}

\section{Convergence of optimal control locations}\label{approximation of opt control}

In practice, the integral Riccati equation in an infinite-dimensional space cannot be solved directly. Usually, we approximate and solve it in finite-dimensional space by a sequence of approximations from various numerical methods. Let ${X_n}$ be a family of finite-dimensional subspaces of $X$ and $\mathbf P_n$ be the corresponding orthogonal projection of $X$ onto $X_n$. The finite-dimensional spaces $\{X_n\}$ inherit the norm from
$X$. For every $n\in\mathds N$, let $T_n(\cdot, \cdot)$ be a mild evolution operator on $X_n$, $B_n(t)\in L^\infty_s(t_0, b;U,X_n)$ and $C_n(t)=C(t)\mathbf P_n$, $G_n\in \mathcal L(X_n)$.
This defines a sequence of approximations
\begin{equation*}
x(t)=T_n(t,t_0)x(t_0)+\int_{t_0}^{t}T_n(t,s)B_n(s)u(s)ds, \quad t\in [t_0, b]
\end{equation*}
with the cost functional
\begin{equation*}
J_n(t,x,u)=\langle x(b), G_nx(b)\rangle+\int_{t}^{b}\langle C_n(s)x(s), C_n(s)x(s)\rangle+\langle u(s), F(s)u(s)\rangle ds.
\end{equation*}

We denote the optimal control of the approximation by $u_n(t)=-L_n(t)\mathbf P_nx(t)$, $t\in [t_0,b]$, where $L_n(t)=F^{-1}B_n^*\Pi_n$, the perturbed evolution operator of $T_n(\cdot,\cdot)$ by $-B_nL_n$ by $T_{L_n}(\cdot,\cdot)$ and the Riccati operator of the approximation by $\Pi_n$.

In order to guarantee that $\Pi_n(t)$ converges to $\Pi(t)$, the following assumptions are needed in the approximation of control problem for partial differential equations \cite{Gibson79}. For each $x\in X$, $u\in U$, $y\in Y$, when $n\rightarrow\infty$,\\
(a1) $(i)\quad  T_n(t,s)\mathbf P_nx\rightarrow T(t,s)x$; \qquad $(ii)\quad  T_n^*(t,s)\mathbf P_nx\rightarrow T^*(t,s)x$\\
and $\sup_{n}\Vert T_n(t,s)\Vert<\infty$, $(t,s)\in \Gamma_{t_0}^{b}$.\\
(a2) $(i)\quad  B_n(t)u\rightarrow B(t)u;$ \qquad \qquad\quad $(ii)\quad B_n^*(t)\mathbf P_nx\rightarrow B^*(t)x, \  a.e..$\\
(a3) $(i)\quad C_n(t)\mathbf P_nx\rightarrow C(t)x;$ \qquad \qquad $(ii)\quad C_n^*(t)y\rightarrow C^*(t)y, \  a.e..$ \\
(a4) $\sup_{n}\Vert G_n\Vert<\infty$ and
$
G_n\mathbf P_nx\rightarrow Gx.
$

Before we study the uniform convergence from $\Pi_n(t)$ to $\Pi(t)$, we study under which condition the compactness of $\Pi(t)$ can be guaranteed. The following lemma shows this.
\begin{lemma}
We consider the time-varying system \eqref{system} with the cost functional \eqref{cost}. If $B(t)$, $C(t)$, $t\in[t_0,b]$ and $G$ are compact operators, then the unique solution $\Pi(t)$ of the integral Riccati equation \eqref{ire2} is compact.
\end{lemma}
\begin{proof}
Denote $S=C^*C+\Pi BF^{-1}B^*\Pi$,
\begin{equation*}
\Pi(t)=T_L^*(b,t)GT_L(b,t)+\int_t^{b}T_L^*(s,t)S(s)T_L(s,t)ds.
\end{equation*}

Since $B(t)$, $C(t)$ and $G$ are compact,
$T_L^*(b,t)GT_L(b,t)$
and
$T_L^*(s,t)S(s)T_L(s,t)$, $(s,t)\in \Gamma_{t_0}^{b}$
are compact.
Let us only consider the integral part of $\Pi(t)$ firstly. It is clear that there exists a set of orthogonal projections $\{\mathbf P_n\}$ to some finite-dimensional spaces $X_n$, $n\in \mathds N$ such that
$
\lim_{n\rightarrow\infty}\Vert \mathbf P_nT_L^*(s,t)S(s)T_L(s,t)-T_L^*(s,t)S(s)T_L(s,t)\Vert=0.
$
Then, since $T_L(\cdot,\cdot)$ and $S(\cdot)$ are uniformly bounded, it is easy to obtain $\mathbf P_nT_L^*ST_L$ is also uniformly bounded in any time and $n$.
By the dominated convergence theorem,
\begin{equation*}
\lim_{n\rightarrow\infty}\Vert \int_{t}^{b}\mathbf P_nT_L^*(s,t)S(s)T_L(s,t)ds-\int_t^{b}T_L^*(s,t)S(s)T_L(s,t)ds\Vert=0.
\end{equation*}

Obviously, $\int_{t}^{b}\mathbf P_nT_L^*(s,t)S(s)T_L(s,t)ds$ is still finite-rank operator and bounded, so it is compact.

Therefore,
$\int_t^{b}T_L^*(s,t)S(s)T_L(s,t)ds$ is compact. Further, $\Pi(t)$ is compact.
\end{proof}

The following theorem shows the uniform convergence of $\Pi_n(t)$.
\begin{theorem}\label{approximation convergence}
For the sequence of approximations under the assumptions $(a1)-(a4)$, if $B(t)$, $C(t)$, $t\in[t_0,b]$ and $G$ are compact operators and $\lim_{n\rightarrow \infty}\Vert B_n-\mathbf P_n B\Vert_\infty=0$, then
$$\lim_{n\rightarrow \infty}\Vert \Pi_n(t)\mathbf P_n-\Pi(t)\Vert=0, \quad t\in[t_0,b].$$
\end{theorem}
\begin{proof}
From $\lim_{n\rightarrow \infty}\Vert B_n-\mathbf P_nB\Vert_\infty=0$ and $\sup_{t\in[t_0,b]}\Vert B(t)\Vert<\infty$, we have
$$\sup_{n\in \mathds N, \\t\in[t_0,b]}\Vert B_n(t)\Vert<\infty.$$
Moreover, because $B(t)$ is compact and $\mathbf P_n$ is strongly convergent to the identity operator $I$,
$
\lim_{n\rightarrow \infty}\Vert\mathbf  P_nB(t)-B(t)\Vert=0,\ t\in[t_0,b].
$
Further,
\begin{equation*}
\Vert B_n(t)-B(t)\Vert\leqslant \Vert B_n(t)-\mathbf P_nB(t)\Vert+\Vert\mathbf P_nB(t)-B(t)\Vert\rightarrow 0, \quad t\in[t_0,b], \quad n\rightarrow\infty.
\end{equation*}

Meanwhile, by the uniform boundedness of $\Vert T_n(\cdot,\cdot)\Vert$, $\Vert C_n\Vert_\infty$ and $\Vert G_n\Vert$ and \cite[Theorem 5.1]{Gibson79},
for any $x\in X$,
\begin{equation*}
\lim_{n\rightarrow\infty}\Vert\Pi_n(t)x-\Pi(t)x\Vert=0,\quad t\in[t_0,b].
\end{equation*}

Similar to the proof of the uniform boundedness of $\Pi_r$ in Theorem \ref{ex operator}, for the approximations with arbitrary feedback control 
$\tilde u_n(t)=\tilde{L}_n(t)x(t)=\tilde{L}(t)\mathbf P_nx(t)$, $\tilde{L}\in L^\infty_s(t_0, b;X,U)$,
there exists $\lambda_\Pi>0$, such that
$
\sup_{n}\Vert \Pi_n\Vert_\infty<\lambda_\Pi.
$

To proof the uniform convergence of $\Pi_n(t)$, we define $S_n=C_{n}^*C_n+\Pi_nB_nF^{-1}B_{n}^*\Pi_n$ and $S$ with the similar way, then
\begin{align*}
&\Vert \Pi_n(t)\mathbf P_n-\Pi(t) \Vert\leqslant\Vert (T_{L_n}^*(b,t)-T_{L}^*(b,t))G_n\mathbf P_n\Vert\Vert T_{L_n}(b,t)\mathbf P_n\Vert\\
&+\Vert T_{L}^*(b,t)\Vert\Vert (G_n\mathbf P_n-G)T_{L_n}(b,t)\mathbf P_n\Vert+\Vert T_{L}^*(b,t)\Vert\Vert G (T_{L_n}(b,t)\mathbf P_n-T_L(b,t))\Vert\\
&+\int_t^{b}\Vert[T_{L_n}^*(s,t)-T_{L}^*(s,t)]S_n(s)\mathbf P_n\Vert \Vert T_{L_n}(s,t)\mathbf P_n\Vert ds\\
&+\int_t^{b}\Vert T_{L}^*(s,t)\Vert \Vert S_n(s)\mathbf P_n-S(s)\Vert \Vert T_{L_n}(s,t)\mathbf P_n\Vert ds\\
&+\int_t^{b}\Vert T_{L}(s,t)\Vert \Vert S(s)(T_{L_n}(s,t)\mathbf P_n-T_L(s,t))\Vert ds.
\end{align*}

As a result of the uniform boundedness of $\Vert T_n(\cdot,\cdot)\Vert$, $\Vert \Pi_n\Vert_\infty$ and $\Vert B_n\Vert_\infty$ in $n$, $\Vert L_n\Vert_\infty$ is uniform bounded
and
\begin{eqnarray*}
\Vert L_n^*(t)-\mathbf P_nL^*(t)\Vert&\leqslant&\Vert F^{-1}\Vert_\infty(\Vert \Pi_n(t)\Vert\Vert B_n(t)-\mathbf P_nB(t)\Vert\\
&&+\Vert\Pi_n(t)-\mathbf P_n\Pi(t)\Vert\Vert B(t)\Vert\longrightarrow 0, \quad  r\rightarrow r_0.
\end{eqnarray*}
so
$
\lim_{n\rightarrow\infty}\Vert L_n(t)\mathbf P_n-L(t)\Vert=0
$
and
$
\lim_{n\rightarrow\infty}\Vert B_n(t)L_n(t)\mathbf P_n-B(t)L(t)\Vert=0.
$

According to Lemma \ref{per lemma} and assumption (a1),
\begin{align*}
&\lim_{n\rightarrow\infty}\Vert T_{L_n}(t,s)\mathbf P_nx-T_L(t,s)x\Vert=0,\\
&\lim_{n\rightarrow\infty}\Vert T_{L_n}^*(t,s)\mathbf P_nx-T_L^*(t,s)x\Vert=0, \quad  x\in X.
\end{align*}
Finally, because of the compactness of the self-adjoint operator $G_n$ and $G$, we have
$$
\lim_{n\rightarrow\infty}\Vert(T_{L_n}(t,s)-T_L(t,s))G_n\mathbf P_n\Vert=0
$$
and
$
\lim_{n\rightarrow\infty}\Vert G(T_{L_n}(t,s)\mathbf P_n-T_L(t,s))\Vert=0.
$
Meanwhile,
$\Vert S_n\Vert_\infty\leqslant\Vert C_{n}^*C_n\Vert_\infty+\Vert\Pi_nB_nF^{-1}B_{n}^*\Pi_n\Vert_\infty<\infty,  \quad n\in\mathds N$.
Since $C_n=C\mathbf P_n$ is compact,
\begin{eqnarray*}
&&\Vert S_n(t)\mathbf P_n-S(t)\Vert \\&\leqslant& \Vert C_{n}^*(t)C_n(t)-C^*(t)C(t)\Vert+\Vert L_n^*(t)F(t)L_n(t)\mathbf P_n-L(t)F(t)L(t)\Vert\\
&\leqslant& \Vert C^*(t)\mathbf P_n-C^*(t)\Vert\Vert C_n\Vert_\infty+\Vert C^*\Vert_\infty\Vert C(t)\mathbf P_n-C(t)\Vert\\
&&+\Vert L_n^*(t)-L^*(t) \Vert \Vert F\Vert_\infty\Vert L_n\Vert_\infty+\Vert L^*\Vert_\infty\Vert F\Vert_\infty\Vert L_n(t)\mathbf P_n-L(t) \Vert\rightarrow 0,\quad n\rightarrow\infty.
\end{eqnarray*}

By dominated convergence theorem,
$
 \Vert \Pi_n(t)\mathbf P_n-\Pi(t) \Vert\longrightarrow 0,\quad n\rightarrow\infty.
$
\end{proof}

Next we show that the optimal control locations of approximations converge to the optimal control location of the original system.
\begin{theorem}\label{optimal convergence}
Under the assumptions $(a1)-(a4)$ and further assume $B_{r,n}=\mathbf P_nB_{r}$, $r\in\Omega^m$, if $B_r(t)$, $C(t)$ and $G$, $t\in[t_0,b]$ are compact operators 
and $\lim_{r\rightarrow r_0}\Vert B_{r}-B_{r_0}\Vert=0$,
then
\begin{equation*}
\hat{\ell}_n(t)\rightarrow \hat \ell(t), \quad  \hat{r}_n\rightarrow \hat r,\quad n\rightarrow\infty.
\end{equation*}
\end{theorem}
\begin{proof}
From Theorem \ref{approximation convergence},  $\lim_{n\rightarrow \infty}\Vert \Pi_{r,n}(t)\mathbf P_n-\Pi_{r}(t)\Vert=0$
, $r\in\Omega^m$.

Since  $\lim_{r\rightarrow r_0}\Vert B_{r}-B_{r_0}\Vert_\infty=0$,
\begin{equation*}
\Vert B_{r,n}-B_{r_0,n}\Vert_\infty\leqslant\Vert \mathbf  P_n\Vert\Vert B_{r}-B_{r_0}\Vert_\infty\rightarrow0, \quad r\rightarrow r_0.
\end{equation*}

From Theorem \ref{ex operator}, for any $n\in \mathds N$, there exists $\hat l_n(t)=\inf_{r\in \Omega^m}\Vert \Pi_{r,n}(t)\Vert$.

On one hand,
\begin{eqnarray*}
\hat \ell_n(t)&=&\inf_{r\in \Omega^m}\Vert \Pi_{r,n}(t)\Vert\leqslant\Vert \Pi_{\hat r,n}(t)\Vert\leqslant\Vert \Pi_{\hat r,n}(t)-\Pi_{\hat r}(t)\Vert+\Vert \Pi_{\hat r}(t)\Vert\\
&\rightarrow& \Vert \Pi_{\hat r}(t)\Vert=\hat \ell(t), \quad n\rightarrow\infty,
\end{eqnarray*}
so
$\lim_{n\rightarrow \infty}\sup_{n}\hat l_n(t)\leqslant\hat l.$

On the other hand, there exists a subsequence $\{\hat \ell_{n_k}(t)\}$ such that $\lim_{k\rightarrow\infty}\hat \ell_{n_k}(t)= \lim_{n\rightarrow\infty}\inf_n \hat \ell_n(t),$
where $\hat \ell_{n_k}(t)=\inf_{r\in \Omega^m}\Vert\Pi_{r,n_k}(t)\Vert=\Vert\Pi_{r_{n_k},n_k}(t)\Vert.$
Due to the compactness of $\Omega^m$, without loss of the generality, we assume $\lim_{k\rightarrow\infty}\hat r_{n_k}=\bar r$,
\begin{equation*}
\Vert B_{\hat r_{n_k},n_k}-B_{\bar r}\Vert_\infty
\leqslant\Vert \mathbf P_{n_k}\Vert\Vert B_{\hat r_{n_k}}-B_{\bar r}\Vert_\infty+\Vert\mathbf  P_{n_k}B_{\bar r}-B_{\bar r}\Vert_\infty\longrightarrow 0,\quad k\rightarrow\infty
\end{equation*}
and
\begin{equation}\label{sublocation continuous}
\Vert \Pi_{\hat r_{n_k},n_k}(t)-\Pi_{\bar r}(t)\Vert\leqslant\Vert \Pi_{\hat r_{n_k},n_k}(t)-\Pi_{r_{n_k}}(t)\Vert+\Vert \Pi_{r_{n_k}}(t)-\Pi_{\bar r}(t)\Vert\\
\longrightarrow 0, \quad k\rightarrow\infty.
\end{equation}
Hence,
\begin{equation*}
\lim_{n\rightarrow \infty}\inf_{n}\hat \ell_n(t)=\lim_{k\rightarrow\infty}\hat \ell_{n_k}(t)=\lim_{k\rightarrow\infty}\Vert \Pi_{\hat r_{n_k},n_k}(t)\Vert=\Vert\Pi_{\bar r}(t)\Vert\geqslant\Vert\Pi_{\hat r}(t)\Vert=\hat \ell_r(t),
\end{equation*}
so
$\lim_{n\rightarrow\infty}\hat{\ell}_n(t)=\hat \ell(t).$
Further, $\lim_{n\rightarrow\infty}\hat{\ell}_n(t)=\lim_{n\rightarrow \infty}\inf_{n}\hat \ell_n(t)=\hat \ell(t)$, so
\begin{equation*}
\lim_{k\rightarrow\infty}\Vert \Pi_{\hat r_{n_k},n_k}(t)\Vert=\Vert\Pi_{\bar r}(t)\Vert=\Vert\Pi_{\hat r}(t)\Vert.
\end{equation*}
By the continuity with respect to $r_{n_k}$ in \eqref{sublocation continuous}, $\lim_{k\rightarrow\infty}\hat r_{n_k}=\hat r$.
\end{proof}

For the proof of the uniform convergence of the Riccati operators of the approximations in nuclear norm for stochastic systems, we need the following lemma.
\begin{lemma}
Let G be a nonnegative nuclear operator in a separable Hilbert space $X$ and assume that $T_n$ strongly converges to $T$, $T_n$, $T\in \mathcal L(X)$ are uniformly bounded by $\lambda_T$. Then
\begin{equation*}
\lim_{n\rightarrow\infty}\Vert (T_n-T)G\Vert_1=0.
\end{equation*}
\end{lemma}
\begin{proof}
Assume $\{e_i\}$ is the orthogonal basis in $X$ and there exist a partial isometry $V$ such that $G=V\vert G\vert,$
where $\vert G\vert=(G^*G)^\frac{1}{2}$, then,
\begin{eqnarray*}
\Vert (T_n-T)V\vert G\vert^{\frac{1}{2}}e_i\Vert\leqslant\Vert T_n-T\Vert\Vert V\vert G\vert^{\frac{1}{2}}e_i\Vert\leqslant2\lambda_T\Vert V\vert G\vert^{\frac{1}{2}}e_i\Vert.
\end{eqnarray*}
Because of the strong convergence of $T_n$, $\lim_{n\rightarrow\infty}\Vert (T_n-T)V\vert G\vert^{\frac{1}{2}}e_i\Vert=0$.

Since G is a nuclear operator, then $\vert G\vert^{\frac{1}{2}}$ is a Hilbert-Schmidt operator, so
\begin{equation*}
\sum_{i=1}^\infty\Vert (T_n-T)V\vert G\vert^{\frac{1}{2}} e_i\Vert=2\lambda_T\sum_{i=1}^\infty\Vert V\vert G\vert^{\frac{1}{2}}e_i\Vert<\infty.
\end{equation*}
By the dominated convergence theorem,
\begin{eqnarray*}
&&\lim_{n\rightarrow\infty}\Vert (T_n-T)V\vert G\vert^{\frac{1}{2}}\Vert_{HS}\\&=&\lim_{n\rightarrow\infty}\sum_{i=1}^\infty\Vert (T_n-T)V\vert G\vert^{\frac{1}{2}}e_i\Vert=\sum_{i=1}^\infty\lim_{n\rightarrow\infty}\Vert (T_n-T)V\vert G\vert^{\frac{1}{2}}e_i\Vert=0.
\end{eqnarray*}
Then,
$
\Vert (T_n-T)G\Vert_{1}\leqslant \Vert (T_n-T)V\vert G\vert^{\frac{1}{2}}\Vert_{HS}\Vert \vert G\vert^{\frac{1}{2}}\Vert_{HS}\rightarrow0,\quad n\rightarrow\infty.
$
\end{proof}

Associated with Corollary \ref{coro}, the following theorem guarantees the uniform convergence of the Riccati operators of approximations to the Riccati operator of the original system in nuclear norm.
\begin{theorem}\label{approximation nuclear}
For the sequence of approximations under the assumptions $(a1)-(a4)$, if  $U$ and $Y$ are finite dimensional, $\lim_{n\rightarrow \infty}\Vert B_n-\mathbf P_nB\Vert_\infty=0$, $G$ is nuclear operator and $\lim_{n\rightarrow\infty}\Vert G_n\mathbf P_n-G\Vert_1=0$, then
$$\lim_{n\rightarrow \infty}\Vert \Pi_n(t)\mathbf P_n-\Pi(t)\Vert_1=0.$$
\end{theorem}
\begin{proof}
Defining $\mathcal C_t$ in the same way with Corollary \ref{coro} and define
$\mathcal C_{t,n}$ by substituting $n$ into $r$ in \eqref{definition C},
from Theorem \ref{nuclear}.1, $\Pi_n(t)= T_{L_n}^*(b,t)G_nT_{L_n}(b,t)+\mathcal C_{t,n}^*\mathcal C_{t,n}$ is nuclear.
The same with Theorem \ref{approximation convergence}, we also have the uniform boundedness of $\Vert T_n(\cdot,\cdot)\Vert$, $\Vert \Pi_n\Vert_\infty$ $\Vert B_n\Vert_\infty$, $\Vert L_n\Vert_\infty$ in $n$ and
$
\lim_{n\rightarrow\infty}\Vert L_n(t)\mathbf P_n-L(t)\Vert=0,
$
$
\lim_{n\rightarrow\infty}\Vert T_{L_n}(t,s)\mathbf P_nx-T_L(t,s)x\Vert=0.
$
Hence, similar to Theorem \ref{ex nuclear},
\begin{eqnarray*}
&\Vert\mathcal C_{t,n}-\mathcal C_{t}\Vert_{HS}\leqslant \sum_{j=1}^{p+q}\int_{t}^{b}\Vert (T_{L_n}^*(s, t)-T_{L}(s, t)^*)[C_n^*(s),L_n^*(s)F^{\frac{1}{2}}(s)]e_j]\Vert\\\nonumber
&+\Vert T_{L}^*(s, t)[C_n(s)-C(s),(L_n^*(s)-L^*(s))F^{\frac{1}{2}}(s)]e_j]\Vert ds\longrightarrow0, \ r\rightarrow r_0, \ (s,t)\in\Gamma_{t_0}^{b}.\nonumber
\end{eqnarray*}
Then, since $G$ is nuclear operator with $\lim_{n\rightarrow\infty}\Vert G_n\mathbf P_n-G\Vert_1=0$,
\begin{align*}
&\Vert \Pi_n(t)\mathbf P_n-\Pi(t) \Vert_1\\
&\leqslant\Vert T_{L_n}^*(b,t)\Vert\Vert G_n( T_{L_n}(b,t)\mathbf P_n-T_L(b,t))\Vert_1+\Vert (T_{L_n}^*(b,t)- T_{L}^*(b,t))G_n\Vert_1\Vert T_{L}(b,t)\Vert\\
&+\Vert T_{L}^*(b,t)\Vert\Vert G_n\mathbf P_n-G\Vert_1\Vert T_{L}(b,t)\Vert+\Vert\mathcal C_{t,n}^*-C_{t}^*\Vert_{HS}\Vert\mathcal C_{t,n}\Vert_{HS}\\
&+\Vert C_{t}^*\Vert_{HS}\Vert\mathcal C_{t,n}-C_{t}\Vert_{HS}\rightarrow 0, \ n\rightarrow \infty.
\end{align*}
\end{proof}

\begin{theorem}\label{optimal convergence2}
Under the assumptions $(a1)-(a4)$ and further assume $B_{r,n}=\mathbf P_nB_{r}$, $r\in\Omega^m$, if the input space $U$ and the output space $Y$ are finite dimensional, $\lim_{r\rightarrow r_0}\Vert B_{r}-B_{r_0}\Vert=0$, $G$ is nuclear operator and $\lim_{n\rightarrow\infty}\Vert G_nP_n-G\Vert_1=0$,
then
\begin{equation*}
\hat{\ell}_{1,n}(t)\rightarrow \hat \ell_1(t), \quad \hat{r}_n\rightarrow \hat r,\quad  n\rightarrow\infty.
\end{equation*}
\end{theorem}
\begin{proof}
From Theorem \ref{ex nuclear} and Theorem \ref{approximation nuclear}, we have
$
\lim_{r\rightarrow r_0}\Vert \Pi_{r}(t)-\Pi_{r_0}(t)\Vert_1=0
$
and
$
\lim_{n\rightarrow \infty}\Vert \Pi_{r,n}(t)\mathbf P_n-\Pi_{r}(t)\Vert_1=0.
$
The same with Theorem \ref{optimal convergence}, we have
$
\hat \ell_{1,n}(t)\leqslant \Vert \Pi_{\hat r}(t)\Vert_1=\hat \ell_1(t), \quad n\rightarrow\infty.
$
Besides, there exists a subsequence $\{\hat \ell_{1,n_k}(t)\}$ such that
\begin{equation*}
\lim_{n\rightarrow \infty}\inf_{n}\hat \ell_{1,n}(t)=\lim_{k\rightarrow\infty}\hat \ell_{1,n_k}(t)=\lim_{k\rightarrow\infty}\Vert \Pi_{\hat r_{n_k},n_k}(t)\Vert_1=\Vert\Pi_{\bar r}(t)\Vert_1\geqslant\Vert\Pi_{\hat r}(t)\Vert_1=\hat \ell_1(t).
\end{equation*}

Therefore,
$
\lim_{n\rightarrow\infty}\hat{\ell}_{1,n}(t)=\hat \ell_1(t)
$
and
$
\lim_{k\rightarrow\infty}\Vert \Pi_{\hat r_{n_k},n_k}(t)\Vert_1=\Vert\Pi_{\hat r}(t)\Vert_1.
$
By the continuity in Theorem \ref{ex nuclear},
$
\lim_{k\rightarrow\infty}\hat r_{n_k}=\hat r.
$
\end{proof}

\section{Kalman filter in Hilbert spaces and the duality of LQ optimal control problem}\label{kalman filter}
There are several literatures \cite{Catlin09}, \cite{Curtain78}, \cite{Falb67}, \cite{Kalman60} discussing the Kalman filter in different approaches and the duality to the linear-quadratic optimal control. 
However, to author's knowledge, these derivations involve the generator of semigroups or evolution operators.
In this section, without the differentiability of evolution operators, we first derive the Kalman filter in real separable Hilbert spaces. 
Further, we will discuss the duality between Kalman filter and linear-quadratic optimal control.

Let $(\Omega, \mathcal B, \mu)$ be a complete probability space and $\mathcal X$, $\mathcal E$, $\mathcal Y$ be a real separable Hilbert spaces. First, we define some basic concepts of probability theory in Hilbert spaces \cite{Curtain78}, \cite{Parthasarathy67}.

\begin{definition}
The map $x: \Omega\rightarrow \mathcal X$ is a \emph{$\mathcal X-$valued random variable} if it is strong measurable with respect to a measure $\mu$.
\end{definition}

\begin{definition}
$\mu$ is a \emph{totally finite measure} on $\mathcal X$ if for any $\mathcal X-$valued random variable $x$,
$\int_{\Omega}\Vert x\Vert d\mu<\infty$.
Further, if there exists $ \bar x \in \mathcal X$ such that
\begin{equation*}
\langle \bar x,h\rangle =\mathbf E\langle \bar x, h\rangle=\int_{\Omega} \langle \bar x, h\rangle \mu(dx), \quad \forall h\in \mathcal X,
\end{equation*}
$\bar x$ is called the \emph{mean} or \emph{expectation} of $x$ and denoted by $\mathbf Ex$.
\end{definition}

\begin{definition}
For any $\mathcal X-$valued random variable $x$ with mean $\mathbf Ex$, the \emph{covariance operator} $P$ of $x$, also denoted by $Cov(x)$, if it exists, is given by
\begin{equation*}
\langle Ph_1,h_2\rangle=\langle h_1,Ph_2\rangle=\int_{\Omega} \langle x-\mathbf Ex, h_1\rangle \langle x-\mathbf Ex, h_2\rangle \mu(dx), \quad\forall h_1, h_2\in \mathcal X.
\end{equation*}
\end{definition}

\begin{definition}
The random variables $x, y$ whose expectations exist are \emph{independent} if
$\mathbf E(\langle x,y\rangle)=\langle \mathbf E(x), \mathbf E(y)\rangle$.
\end{definition}

\begin{definition}
 Let $\mu$ be a probability measure on $\mathcal X$. If for any $x\in \mathcal X$, the random variable $\langle x, \cdot\rangle$ has a Gaussian distribution, then $\mu$ is called a
\emph{Gaussian measure}. Further, we denote $x$ of the Gaussian measure with mean $\bar x$ and covariance $P$ by $x\sim N(\bar x, P)$.
\end{definition}

\begin{definition}
$\{\omega(t), t\in \mathds R\}$ is a set of \emph{white noises} if for any $t\in[0,+\infty]$, there exists a covariance operator $W(t)$ such that $\omega(t)\sim N(0, W(t))$ and for any $t\neq s$,
$\omega(t)$ and $\omega(s)$ are independent.
\end{definition}

We consider time-varying systems on Hilbert spaces given by
\begin{equation}\label{state}
x(t)=M(t,t_0)x(t_0)+\int_{t_0}^{t}M(t,s)[B(s)u(s)+D(s)\omega(s)]ds, \quad (t,t_0)\in\Gamma_{t_0}^{b},
\end{equation}
where $M(\cdot, \cdot)$ is a mild evolution operator on $\mathcal X$. $x(t)$ and $\omega(t)$ are random variables with values in $\mathcal X$ and $\mathcal E$, respectively and $\omega(t)\sim N(0,W(t))$ is the white noise. Further, we assume $u\in L^2(t_0,b;U)$, $B\in L^\infty_s(t_0, b;U, \mathcal X)$, $B^*\in L^\infty_s(t_0, b; \mathcal X, U )$, $D\in L^\infty_s(t_0, b; \mathcal E, \mathcal X )$.

We consider the following observation system
\begin{equation}\label{observation}
y(t)=H(t)x(t)+E(t)\nu(t), \quad t\in [t_0,b],
\end{equation}
where $H\in L^\infty_s(t_0, b; \mathcal X, \mathcal Y )$, $E\in L^\infty_s(t_0, b; \mathcal E, \mathcal Y )$, $y(t)$ and $\nu(t)$ are random variables with values in $\mathcal Y$ and $\mathcal E$, respectively and $\nu(t)\sim N(0,V(t))$ is the white noise and $V(t)$ is a coercive operator..

In our paper, we only consider the integral form of time-varying systems.
Let $Y_t=\lbrace y(s), t_0 \leqslant s\leqslant t\rbrace$,  the linear unbiased estimation of the filter problem $\hat x(t\vert t)$ of $x(t)$ \cite{Kalman61} has the form
\begin{eqnarray}\label{estimation}
\hat{x}(t\vert t)&=&M(t,t_0)\hat x(t_0\vert t_{-1})\\
&&+\int_{t_0}^{t}M(t,s)B(s)u(s)ds+\int_{t_0}^{t}K_f(t,s)[y(s)-H(s)\hat{x}(s\vert s)]ds, \nonumber
\end{eqnarray}
where $\hat x(t_0\vert t_{-1})=\mathbf E(x(t_0))$, $Cov(x(t_0))=P(t_0\vert t_{-1})$ and $K_f(\cdot,\cdot)$ is an unknown linear gain operator.

Denoting $\tilde x(t\vert t):=x(t)-\hat x(t\vert t)$,  $P(t\vert t):=Cov(\tilde x(t))$ and $R(t):=E(t)V(t)E^*(t)$, we obtain the following theorem
\begin{theorem}\label{kalman gain}
For the time-varying system \eqref{state} with the observation system \eqref{observation}, the linear unbiased estimation of the filter problem $\hat x(t\vert t)$ of $x(t)$ is optimal if the linear gain operator in \eqref{estimation} is given by
$K_f(t,\tau)=M(t,\tau) P(\tau \vert \tau)H^*(\tau)R^{-1}(\tau)$, $\tau\leqslant t$.
\end{theorem}
\begin{proof}
By Wiener-Hopf's equation \cite{Falb67}, \cite{Kalman61},  $\hat x(t\vert t)$ minimizes the minimal covariance if and only if
$\mathbf E\langle \tilde x(t), h_1\rangle \langle y(\tau)-H(\tau)\hat x(\tau\vert \tau), h_2\rangle=0,\quad \tau< t$, $h_1$, $h_2\in \mathcal X$. Further, according to \cite[Corollary 6.3]{Curtain78},
$\mathbf E\langle\hat x(t\vert t),h_1\rangle\langle \tilde x(t\vert t), h_2\rangle=0$.
Hence, on one hand,
\begin{eqnarray*}
&&\mathbf E\langle x(t),h_1\rangle\langle y(\tau)-H(\tau)\hat x(\tau\vert \tau), h_2\rangle\\
&=&\mathbf E\langle x(t),h_1\rangle\langle H(\tau)\tilde x(\tau\vert \tau),h_2\rangle+\mathbf E\langle x(t),h_1\rangle\langle E(\tau)\nu(\tau),h_2\rangle\\
&=&\mathbf E\langle M(t,\tau)x(\tau),h_1\rangle\langle H(\tau)\tilde x(\tau\vert \tau),h_2\rangle-\mathbf E\langle M(t,\tau)\hat x(\tau\vert \tau),h_1\rangle\langle H(\tau)\tilde x(\tau\vert\tau),h_2\rangle\\
&=& \mathbf E\langle M(t,\tau) \tilde x(t\vert t),h_1\rangle\langle H(\tau)\tilde x(\tau\vert \tau),h_2\rangle=\langle h_1, M(t,\tau)P(\tau \vert \tau)H^*(\tau)h_2\rangle.
\end{eqnarray*}
On the other hand,
\begin{eqnarray*}
&&\mathbf E\langle \hat x(t\vert t),h_1\rangle\langle y(\tau)-H(\tau)\hat x(\tau\vert \tau),h_2\rangle\\
&=&\mathbf E\langle \int_{\tau}^{t}K_f(t,s)[H(s)\tilde{x}(s\vert s)+E(s)\nu(s)]ds,h_1\rangle\langle y(\tau)-H(\tau)\hat x(\tau\vert \tau),h_2\rangle\\
&=&\mathbf E\langle \int_{\tau}^{t}K_f(t,s)E(s)\nu(s)ds,h_1\rangle\langle H(\tau)\tilde x(\tau\vert \tau)+ E(\tau)\nu(\tau),h_2\rangle\\
&=& \mathbf E\langle \int_{\tau}^{t}K_f(t,s)E(s)\nu(s)ds,h_1\rangle\langle E(\tau)\nu(\tau),h_2\rangle=\langle h_1,K_f(t,\tau)R(\tau)h_2\rangle
\end{eqnarray*}

Therefore,
$
K_f(t,\tau)R(\tau)= M(t,\tau)P(\tau \vert \tau)H^*(\tau).
$
Since $R(t)$ is coercive, we obtain
$$
K_f(t,\tau)= M(t,\tau)P(\tau \vert \tau)H^*(\tau)R^{-1}(\tau),\quad \tau<t.
$$
If $t=\tau$, by the strong continuity of $K_f(t,\cdot)$, $K_f(t,t)= P(t \vert t)H^*(t)R^{-1}(t)$.
\end{proof}

Defining $K(t):= K_f(t,t)=P(t \vert t)H^*(t)R^{-1}(t)$, Theorem \ref{kalman gain} implies that
\begin{eqnarray}\label{diff estimation}
\tilde{x}(t\vert t)&=&M(t,t_0)\tilde x(t_0\vert t_{-1})-\int_{t_0}^tM(t,s)K(s)H(s)\tilde x(s\vert s)ds\nonumber\\
&&+\int_{t_0}^tM(t,s)[D(s)\omega(s)-K(s)E(s)\nu(s)]ds.
\end{eqnarray}

\begin{theorem}
Equation \eqref{diff estimation} is equivalent to
\begin{equation}\label{diff x(t)}
\tilde{x}(t\vert t)=M_K(t,t_0)\tilde x(t_0\vert t_{-1})+\int_{t_0}^{t}M_K(t,s)\left(D(s)\omega(s)-K(s)E(s)\nu(s)\right)ds,
\end{equation}
where $
M_K(t,\tau)x=M(t,\tau)x-\int_\tau^{t}M_K(t,s)K(s)H(s)M(s,\tau)xds$, $(t,\tau)\in\Gamma_{t_0}^b$
\end{theorem}
\begin{proof}
From \eqref{diff estimation},
\allowdisplaybreaks
\begin{eqnarray*}
&&\tilde{x}(t\vert t)\\
&=&M_K(t,t_0)\tilde x(t_0\vert t_{-1})+\int_{t_0}^tM_K(t,s)K(s)H(s)M(s,t_0)\tilde x(t_0\vert t_{-1})ds\\
&&-\int_{t_0}^tM_K(t,s)K(s)H(s)\tilde x(s\vert s)ds\\
&&-\int_{t_0}^t\int_s^tM_K(t,\eta)K(\eta)H(\eta)M(\eta,s)K(s)H(s)\tilde x(s\vert s)d\eta ds\\
&&+\int_{t_0}^tM_K(t,s)[D(s)\omega(s)-K(s)E(s)\nu(s)]ds\\
&&+\int_{t_0}^t\int_s^tM_K(t,\eta)K(\eta)H(\eta)M(\eta,s)\left(D(s)\omega(s)-K(s)E(s)\nu(s)\right)d\eta ds\\
&=&M_K(t,t_0)\tilde x(t_0\vert t_{-1})+\int_{t_0}^tM_K(t,s)\left(D(s)\omega(s)-K(s)E(s)\nu(s)\right)ds\\
&&-\int_{t_0}^tM_K(t,s)K(s)H(s)\tilde x(s\vert s)ds+\int_{t_0}^tM_K(t,s)K(s)H(s)M(s,t_0)\tilde x(t_0\vert t_{-1})ds\\
&&-\int_{t_0}^tM_K(t,s)K(s)H(s)\int_{t_0}^sM(s,\eta)K(\eta)H(\eta)\tilde x(\eta\vert\eta)d\eta ds\\
&&+\int_{t_0}^tM_K(t,s)K(s)H(s)\int_{t_0}^sM(s,\eta)\left(D(\eta)\omega(\eta)-K(\eta)E(\eta)\nu(\eta)\right)d\eta ds\\
&=&M_K(t,t_0)\tilde x(t_0\vert t_{-1})+\int_{t_0}^tM_K(t,s)\left(D(s)\omega(s)-K(s)E(s)\nu(s)\right)ds.
\end{eqnarray*}
\end{proof}

For finite-dimensional systems, the trace of the covariance of $\tilde x(t\vert t)$ is considered as an evaluation of the estimation errors. For systems on Hilbert spaces, similarly we consider the nuclear norm of the covariance of $\tilde x(t\vert t)$. Defining $Q(t):=D(t)W(t)D^*(t)$, we obtain the following theorem.
\begin{theorem}
The covariance (if exists) of $\tilde x(t\vert t)$ satisfies the IRE
\begin{eqnarray}\label{cov kf}
P(t \vert t)&=&M_K(t,t_0)P(t_0\vert t_{-1})M_K^*(t,t_0)\\
&+&\int_{t_0}^tM_K(t,s)\left[Q(s)+P(s\vert s)H^*(s)R^{-1}(s)H(s)P(s\vert s)\right]M_K^*(t,s)ds.\nonumber
\end{eqnarray}
\end{theorem}
\begin{proof}
For $\tilde x(t\vert t)$ in \eqref{diff x(t)}, assume its covariance $P(t \vert t)$ exists and define $\mathcal Q_t$: $L^2(t_0,t;\mathcal E\times \mathcal E)\rightarrow \mathcal X$ by
\begin{equation*}
 \mathcal Q_t\left(\begin{array}{cc}
              \omega\\
              \nu
             \end{array}\right)
=\int_{t_0}^t\left[M_K(t,s)D(s), -M_K(t,s)K(s)E(s)\right]\left(\begin{array}{cc}
              \omega(s)\\
              \nu(s)
             \end{array}\right)ds.
\end{equation*}
Its adjoint operator $\mathcal Q_t^*$: $\mathcal X\rightarrow L^2(t_0,t;\mathcal E\times\mathcal E)$ is given by
\begin{equation*}
 \mathcal Q_t^*x=\left(\begin{array}{cc}
              D^{*}(\cdot)M_{K}^*(t,\cdot)\\
              -E^*(\cdot)K^*(\cdot)M_K^*(t,\cdot)
             \end{array}\right)x, \quad x\in \mathcal X.
\end{equation*}

Then, we obtain
\allowdisplaybreaks
\begin{eqnarray*}
&& \mathbf E\langle \tilde x(t\vert t),h_1\rangle\langle\tilde x(t\vert t),h_2\rangle\\
&=&\mathbf E\langle M_K(t,t_0)\tilde x(t_0\vert t_{-1}),h_1\rangle\langle M_K(t,t_0)\tilde x(t_0\vert t_{-1}),h_2\rangle\\
&&+\mathbf E\langle  \mathcal Q_t\left(\begin{array}{cc}
              \omega\\
              \nu
             \end{array}\right), h_1\rangle\langle\mathcal Q_t\left(\begin{array}{cc}
              \omega\\
              \nu
             \end{array}\right),h_2\rangle\\
&=&\langle M_K(t,t_0)P(t_0\vert t_{-1})M_K^*(t,t_0)h_1,h_2\rangle+\langle \mathcal Q_tCov\left(\left(\begin{array}{cc}
              \omega\\
              \nu
             \end{array}\right)\right) \mathcal Q_t^*h_1, h_2\rangle\\
&=&\langle M_K(t,t_0)P(t_0\vert t_{-1})M_K^*(t,t_0)h_1,h_2\rangle\\
&&+\langle\int_{t_0}^tM_K(t,s)\left[Q(s)+K(s)R(s)K^*(s)\right]M_K^*(t,s)h_1ds,h_2\rangle,
\end{eqnarray*}
Hence, for any $x\in \mathcal X$
\begin{eqnarray*}
P(t \vert t)x&=&M_K(t,t_0)P(t_0\vert t_{-1})M_K^*(t,t_0)x\\
&&+\int_{t_0}^tM_K(t,s)\left[Q(s)+K(s)R(s)K^*(s)\right]M_K^*(t,s)xds\\
&=&M_K(t,t_0)P(t_0\vert t_{-1})M_K^*(t,t_0)x\\
&&+\int_{t_0}^tM_K(t,s)\left[Q(s)+P(s\vert s)H^*(s)R^{-1}(s)H(s)P(s\vert s)\right]M_K^*(t,s)xds.
\end{eqnarray*}
\end{proof}

A comparison with the main results of the linear-quadratic optimal control problem in Section \ref{main results of lq} yields: By observing the similarity between \eqref{cov kf} and the second integral Riccati equation related to the linear quadratic
optimal control problem, it is clear that to consider the covariance of $\tilde x(t\vert t)$ of the time-varying system \eqref{state}
with the observations \eqref{observation}
is equivalent to consider the Riccati operator $\Pi(b-t)$ in \eqref{ire2} corresponding to the time-varying system
\begin{equation*}
x(t)=T(t,t_0)x(t_0)+\int_{t_0}^t T(t,s)B(s)u(s)ds.
\end{equation*}
with the cost functional
\begin{equation*}
J(t,x, u)=\langle x(b), Gx(b)\rangle+\int_{t}^{b}\langle C(s)x(s),C(s)x(s)\rangle+\langle u(s), F(s)u(s)\rangle ds,
\end{equation*}
where $T(t,s)=M^*(b-s,b-t)$, $B(s)=H^*(b-s)$, $G=P(t_0\vert t_{-1})$, $C(s)= Q^{\frac{1}{2}}(b-s)$, $F(s)=R(b-s)$, $(t,s)\in \Gamma_{t_0}^b$.

Then, by the duality between the linear quadratic control problem and Kalman filters, Corollary \ref{coro} implies the following condition to guarantee the existence and nuclearity of $P(t\vert t)$.
\begin{theorem}\label{nuclear kf}
For the time-varying system \eqref{state} with the observation system \eqref{observation}, if $\mathcal E$ and  $\mathcal Y$ are finite dimensional and $P(t_0\vert t_{-1})$ is a nuclear operator, then the covariance of $\tilde x(t\vert t)$ based on $Y_t$ satisfying \eqref{cov kf} exists and is a nuclear operator.
\end{theorem}

\section{Kalman smoother in Hilbert spaces}\label{ks on hilbert}
\allowdisplaybreaks
In this section, we study the optimal linear unbiased estimation of $x(\tau)$ based on $Y_t$ by $\hat x(\tau\vert t)$, $\tau\leqslant t$.
We still constrain the linear estimation of $x(\tau\vert t)$ has the form
\begin{equation}\label{linear ks}
\hat x(\tau\vert t)=\int_{t_0}^tK_s(\tau,s)[y(s)-H(s)\hat x(s|s)]ds, \quad \tau\leqslant t,
\end{equation}
where $K_s(\cdot,\cdot)$ is an unknown linear operator.

Since in the case $\tau=t$, \eqref{linear ks} with the minimal covariance is equivalent to the optimal linear unbiased estimation based on Kalman filter, in order to determine the optimal estimation of $\hat x(\tau\vert t)$, $\tau\leqslant t$, we can rewrite \eqref{linear ks} as
\begin{equation}\label{estimation smoother}
\hat x(\tau\vert t)=\hat x(\tau\vert \tau)+\int_{\tau}^tK_s(\tau,s)[y(s)-H(s)\hat x(s|s)]ds.
\end{equation}

\begin{theorem}\label{gain of ks}
For the time-varying system \eqref{state} with the observation system \eqref{observation},  the linear unbiased estimation of the filter problem $\hat x(\tau\vert t)$ of $x(\tau)$ is optimal if $K_s(\cdot,\cdot)$ in \eqref{estimation} is given by
$$
K_s(\tau,\eta)=P(\tau\vert \tau)M_K^*(\eta,\tau)H^*(\eta)R^{-1}(\eta),\quad \tau\leqslant \eta\leqslant t.
$$
\end{theorem}
\begin{proof}
By Wiener-Hopf's equation \cite{Kalman61}, \cite{Falb67},
$E\langle\tilde x(\tau\vert t),h_1\rangle\langle y(\eta)-H(\eta)\hat x(\eta|\eta), h_2\rangle=0$, $h_1\in \mathcal X, h_2\in \mathcal Y$, for any $\eta<t$.
It is clear for any $\eta<\tau$,  $E\langle\tilde x(\tau\vert t),h_1\rangle\langle y(\eta)-H(\eta)\hat x(\eta|\eta), h_2\rangle=0$ holds.
Now we assume $\tau\leqslant\eta<t$. On one hand,
\begin{eqnarray*}
&&\mathbf E\langle x(\tau),h_1\rangle\langle y(\eta)-H(\eta)\hat x(\eta|\eta),h_2\rangle\\
&=&\mathbf E\langle x(\tau)-\hat x(\tau|\eta),h_1\rangle\langle H(\eta)\tilde x(\eta\vert \eta),h_2\rangle\\
&=&\mathbf E\langle \tilde x(\tau\vert \tau)-\int_{\tau}^{\eta}K_s(\tau,s)[y(s)-H(s)\hat x(s\vert s)]ds,h_1\rangle\langle H(\eta)\tilde x(\eta\vert \eta),h_2\rangle\\
&=&\mathbf E\langle \tilde x(\tau\vert \tau), h_1\rangle\langle H(\eta)M_K(\eta,\tau)H^*(\eta)\tilde x(\tau\vert \tau),h_2\rangle=\langle h_1, P(\tau\vert \tau)M_K^*(\eta,\tau)H^*(\eta)h_2\rangle.
\end{eqnarray*}
On the other hand,
\begin{eqnarray*}
&&\mathbf E\langle \hat x(\tau|t),h_1\rangle\langle  y(\eta)-H(\eta)\hat x(\eta|\eta),h_2\rangle\\
&=&\mathbf E\langle \hat x(\tau|\eta),h_1\rangle\langle H(\eta)\tilde x(\eta|\eta)+E(\eta)\nu(\eta),h_2\rangle\\
&&+\mathbf E\langle \int_{\eta}^tK_s(\tau,s)[y(s)-H(s)\hat x(s)]ds, h_1\rangle
\langle   y(\eta)-H(\eta)\hat x(\eta|\eta),h_2\rangle\\
&=&\mathbf E\langle \int_{\tau}^tK_s(s,\tau)E(s)\nu(s)ds, h_1\rangle\langle E(\eta)\nu(\eta),h_2\rangle=\langle h_1, K_s(\tau,\eta)R(\eta)h_2\rangle.
\end{eqnarray*}
By the coercivity of $R(t)$, we obtain
$
K_s(\tau,\eta)=P(\tau\vert \tau)M_K^*(\eta,\tau)H^*(\eta)R^{-1}(\eta).
$
\end{proof}

Defining $\tilde x(\tau \vert t)=x(\tau)-\hat x(\tau\vert t)$, Theorem \ref{gain of ks} implies
\begin{eqnarray*}
\tilde x(\tau \vert t)=\tilde x(\tau\vert \tau)-P(\tau\vert \tau)\int_{\tau}^tM_K^*(s,\tau)H^*(s)R^{-1}(s)[y(s)-H(s)\hat x(s\vert s)]ds.
\end{eqnarray*}

Thus, its covariance can be derived by
\begin{theorem}
The covariance (if exists) of $\tilde x(\tau\vert t)$, $(t,\tau)\in\Gamma_{t_0}^{b}$ is
\begin{align}\label{kalman smoother}
&P(\tau\vert t)x=P(\tau\vert \tau)x\\&-P(\tau\vert \tau)\int_{\tau}^tM_K^*(s,\tau)H^*(s)R^{-1}(s)H(s)M_K(s,\tau)P(\tau\vert \tau)xds,\quad x\in \mathcal X.\nonumber
\end{align}
\end{theorem}

\begin{proof}
Denoting the covariance of $\tilde x(\tau\vert t)$ by $P(\tau\vert t)$, we obtain
\allowdisplaybreaks
\begin{align*}
&\langle h_1, P(\tau\vert t)h_2\rangle=\mathbf E\langle\tilde x(\tau\vert t),h_1\rangle\langle\tilde x(\tau\vert t),h_2\rangle\\
&=\mathbf E\langle\tilde x(\tau\vert \tau)-P(\tau\vert \tau)\int_{\tau}^tM_K^*(s,\tau)H^*(s)R^{-1}(s)[y(s)-H(s)\hat x(s)]ds,h_1\rangle\langle\tilde x(\tau\vert t), h_2\rangle\\
&=\mathbf E\langle\tilde x(\tau\vert \tau),h_1\rangle\langle\tilde x(\tau\vert \tau)\\&-P(\tau\vert \tau)\int_{\tau}^tM_K^*(s,\tau)H^*(s)R^{-1}(s)[y(s)-H(s)\hat x(s\vert s)]ds, h_2\rangle\\
&=\mathbf E\langle\tilde x(\tau\vert \tau),h_1\rangle\langle\tilde x(\tau\vert \tau),h_2\rangle\\
&-\mathbf E\langle\tilde x(\tau\vert \tau),h_1\rangle\langle P(\tau\vert \tau)\int_{\tau}^tM_K^*(s,\tau)H^*(s)R^{-1}(s)H(s)\tilde x(s\vert s)ds, h_2\rangle\\
&=\langle h_1, P(\tau\vert \tau)h_2\rangle-E\langle\tilde x(\tau\vert \tau),h_1\rangle\langle P(\tau\vert \tau)\int_{\tau}^tM_K^*(s,\tau)H^*(s)R^{-1}(s)H(s)\\
&\cdot\left(M_K(s,\tau)\tilde x(\tau\vert \tau)+\int_{\tau}^{s}M_K(s,\eta)\left(D(\eta)\omega(\eta)-K(\eta)E(\eta)\nu(\eta)\right)d\eta\right)ds, h_2\rangle\\
&=\langle h_1, P(\tau\vert \tau)h_2\rangle\\
&-\mathbf E\langle\tilde x(\tau\vert \tau),h_1\rangle\langle P(\tau\vert \tau)\int_{\tau}^tM_K^*(s,\tau)H^*(s)R^{-1}(s)H(s)M_K(s,\tau)\tilde x(\tau\vert \tau)ds, h_2\rangle\\
&=\langle h_1, P(\tau\vert \tau)h_2\rangle-\langle h_1,  P(\tau\vert \tau)\int_{\tau}^tM_K^*(s,\tau)H^*(s)R^{-1}(s)H(s)M_K(s,\tau)P(\tau\vert \tau)h_2ds\rangle.
\end{align*}
Hence, for any $x\in \mathcal X$, $(t,\tau)\in\Gamma_{t_0}^{b}$, we get
\begin{equation*}
P(\tau\vert t)x=P(\tau\vert \tau)x- P(\tau\vert \tau)\int_{\tau}^tM_K^*(s,\tau)H^*(s)R^{-1}(s)H(s)M_K(s,\tau)P(\tau\vert \tau)xds.
\end{equation*}
\end{proof}

\begin{theorem}\label{nuclear ks}
For the time-varying system \eqref{state} with the observation system \eqref{observation}, if $\mathcal E$ and  $\mathcal Y$ are finite dimensional and $P(t_0\vert t_{-1})$ is a nuclear operator, then $P(\tau\vert t)$, $(t,\tau)\in\Gamma_{t_0}^{b}$ satisfying \eqref{kalman smoother} exists and is a nuclear operator.
\end{theorem}
\begin{proof}
By Theorem \ref{nuclear kf} and the uniform boundedness of $M_K$, $H$ and $R^{-1}$ in $[t_0,b]$,
\begin{equation*}
\Vert P(\tau\vert t)\Vert_1\leqslant\Vert P(\tau\vert \tau)\Vert_1+\Vert P(\tau\vert \tau)\Vert_1^2\int_{\tau}^t\Vert M_K(s,\tau)\Vert^2 \Vert R^{-1}(s)\Vert\Vert H(s)\Vert^2 ds<\infty,
\end{equation*}
so $P(\tau\vert t)$ is a nuclear operator for any $(t,\tau)\in\Gamma_{t_0}^{b}$.
\end{proof}

\section{Optimal locations of observations based on Kalman filter and smoother}\label{opt loc of kf and ks}

In this section, we also take the observation location problem into account. The location parameter $r$ is defined as in Section \ref{main results of lq}.
The following theorems show the continuity of $P_r(t\vert t)$ and $P_r(\tau\vert t), (t, \tau)\in \Gamma_{t_0}^b$ in nuclear norm.
For the filter problem, due to the duality and Theorem \ref{ex nuclear}, we obtain the following theorem.
\begin{theorem}\label{continuous kf}
Consider the filter problem of the time-varying system \eqref{state} with location-dependent output operators and the observation system \eqref{observation}. If $H_r$ is of the property that $\lim_{r\rightarrow r_0}\Vert H_r-H_{r_0}\Vert_\infty=0$,  $\mathcal E$ and $\mathcal Y$ are finite-dimensional, and $P(t_0\vert t_{-1})$ is nuclear, then
\begin{equation*}
\lim_{r\rightarrow r_0}\Vert P_r(t\vert t)-P_{r_0}(t\vert t)\Vert_1=0, \quad t\in [t_0,b],
\end{equation*}
and there exists an optimal location $\hat{r}^f$ such that,
\begin{equation*}
\hat{\ell}_{1}^f(t)=\Vert P_{\hat{r}^f}(t\vert t)\Vert_1=\inf_{r\in \Omega^m}\Vert P_r(t\vert t)\Vert_1.
\end{equation*}
\end{theorem}

\begin{theorem}\label{continuous ks}
Consider the smoother problem of the time-varying system \eqref{state} with the location-dependent output operators and the observation system \eqref{observation}. 
$H_r$ has the property that $\lim_{r\rightarrow r_0}\Vert H_r-H_{r_0}\Vert_\infty=0$. If $\mathcal E$ and $\mathcal Y$ are finite-dimensional, and $P(t_0\vert t_{-1})$ is nuclear, then,
\begin{equation*}
\lim_{r\rightarrow r_0}\Vert P_r(\tau\vert t)-P_{r_0}(\tau\vert t)\Vert_1=0,  \quad (t,\tau)\in\Gamma_{t_0}^{b},
\end{equation*}
and there exists an optimal location $\hat{r}^s$ such that for any initial time $\tau\in[t_0,b]$, $\tau\leqslant t$,
\begin{equation*}
\hat{\ell}_{1}^s(\tau|t)=\Vert P_{\hat{r}^s}(\tau\vert t)\Vert_1=\inf_{r\in \Omega^m}\Vert P_r(\tau\vert t)\Vert_1.
\end{equation*}
\end{theorem}
\begin{proof}
From Lemma \ref{nuclear ks}, $P_r(\tau\vert t)$, $r\in\Omega^m$ are nuclear operators. Hence,
\allowdisplaybreaks
\begin{align*}
&\Vert P_r(\tau\vert t)-P_{r_0}(\tau\vert t)\Vert_1\leqslant\Vert P_r(\tau\vert \tau)-P_{r_0}(\tau\vert \tau)\Vert_1\\
&+\int_{t_0}^t \Vert P_{r_0}(\tau\vert \tau)M_{K,r_0}^*(s,\tau)H_{r_0}^*(s)-P_r(\tau\vert \tau)M_{K,r}^*(s,\tau)H_r^*(s)\Vert\\
&\cdot\Vert R^{-1}(s)H_{r_0}(s)M_{K,r_0}(s,\tau)P_{r_0}(\tau\vert \tau)\Vert_1 ds+\int_{t_0}^t \Vert P_{r}(\tau\vert \tau)M_{K,r}^*(s,\tau)H_{r}^*(s)R^{-1}(s)\Vert_1\\
&\cdot\Vert H_{r_0}(s)M_{K,r_0}(s,\tau)P_{r_0}(\tau\vert \tau)-H_{r}(s)M_{K,r}(s,\tau)P_{r}(\tau\vert \tau)\Vert ds,
\end{align*}

Since $P_r(t), r\in\Omega^m$ are nuclear operators and $R^{-1}(t), H_{r}(t), M_{K,r_0}(t,\tau)$ are uniformly bounded for $(t,\tau)\in\Gamma_{t_0}^{b}$, then
$
\Vert R^{-1}(s)H_{r_0}(s)M_{K,r_0}(s,\tau)P_{r_0}(\tau\vert \tau)\Vert_1<\infty
$
and so is its adjoint.

By Theorem \ref{continuous kf} and dominated convergence theorem, we obtian
\begin{equation*}
\Vert P_r(\tau\vert t)-P_{r_0}(\tau\vert t)\Vert_1\rightarrow 0,\quad r\rightarrow r_0.
\end{equation*}
Because of the compactness of $\Omega^m$, there exists the optimal location of observations such that
$
\hat{\ell}_{1}^s(\tau|t)=\Vert P_{\hat{r}^s}(\tau\vert t)\Vert_1=\inf_{r\in \Omega^m}\Vert P_r(\tau\vert t)\Vert_1.
$
\end{proof}

Next we consider a sequence of approximations of time-varying systems in order to study the convergence of optimal observation locations based on Kalman filter and smoother. 
Let ${\mathcal X_n}$ be a family of finite-dimensional subspaces of $\mathcal X$ and $\mathbf P_n$ be the corresponding orthogonal projection of $\mathcal X$ onto $\mathcal X_n$. The finite spaces $\{\mathcal X_n\}$ inherit the norm from
$\mathcal X$. For $n\in\mathds N$, let $M_n(\cdot,\cdot)$ be a mild evolution operator on $\mathcal X_n$, $D_n(t)=\mathbf P_nD(t)$ and $H_n(t)=H(t)\mathbf P_n$, $t\in[t_0,b]$.
In order to guarantee that $P_n(t\vert t)$ converges to $P(t\vert t)$, the following assumptions are needed in the approximation of observation problems for partial differential equations. For each $x\in X$, $\omega\in \mathcal E$, $y\in \mathcal Y$\\
(A1) $(i)\quad  M_n(t,s)\mathbf P_nx\rightarrow M(t,s)x;$\quad
    $(ii)\quad  M_n^*(t,s)\mathbf P_nx\rightarrow M^*(t,s)x$\\
and $\sup_{n}\Vert M_n(t,s)\Vert<\infty$, for any $(t,s)\in \Gamma_{t_0}^{b}$.\\
(A2) $(i)\quad   D_n(t)\omega\rightarrow D(t)\omega;$ \quad
    $(ii)\quad   D_n^{*}(t)\mathbf P_nx\rightarrow D^{*}(t)x, \  a.e.\quad  t\in[t_0,b].$\\
(A3) $(i)\quad  H_n(t)\mathbf P_nx\rightarrow H(t)x;$ \quad  $(ii)\quad   H_n^*(t)y\rightarrow H^*(t)y, \  a.e.\quad  t\in[t_0,b].$\\
(A4) $ P_n(t_0\vert t_{-1})\mathbf P_nx\rightarrow P(t_0\vert t_{-1})x$ and $\sup_{n}\Vert P_n(t_0\vert t_{-1})\Vert<\infty$.

The next theorem shows the uniform convergence of the approximations of covariances of the Kalman filter and smoother in nuclear norm.
\begin{theorem}\label{smoother app2}
Assume that the assumptions $(A1)-(A4)$ are satisfied. If $\mathcal E$ and $\mathcal Y$ are finite-dimensional, $\lim_{n\rightarrow \infty}\Vert P_n(t_0\vert t_{-1})\mathbf P_n-P(t_0\vert t_{-1})\Vert_1=0$ and  $P(t_0\vert t_{-1})$ is nuclear, then
\begin{align*}
&\lim_{n\rightarrow \infty}\Vert P_n(t\vert t)\mathbf P_n-P(t \vert t)\Vert_1=0, \\
&\lim_{n\rightarrow \infty}\Vert P_n(\tau\vert t)\mathbf P_n-P(\tau \vert t)\Vert_1=0, \quad (t,\tau)\in \Gamma_{t_0}^b.
\end{align*}
\end{theorem}
\begin{proof}
Due to the duality between Kalman filter and LQ optimal control problem, according to Theorem \ref{approximation nuclear}, we have
\begin{equation}\label{approximation covergence of kf}
\lim_{n\rightarrow \infty}\Vert P_n(t\vert t)\mathbf P_n-P(t \vert t)\Vert_1=0, \quad (t,\tau)\in\Gamma_{t_0}^b.
\end{equation}
Then,
\begin{align*}
&\Vert  P_n(\tau\vert t)\mathbf P_n-P(\tau\vert t)\Vert_1\leqslant\Vert P_n(\tau\vert \tau)\mathbf P_n-P(\tau\vert \tau)\Vert_1\\
&+\int_{\tau}^t \Vert P(\tau\vert \tau)M_{K}^*(s,\tau)H^*(s)-\mathbf P_nP_n(\tau\vert \tau)M_{K,n}^*(s,\tau)H_n^*(s)\Vert\\
&\cdot\Vert R^{-1}(s)H(s)M_{K}(s,\tau)P(\tau\vert \tau)\Vert_1 ds+\int_{\tau}^t \Vert\mathbf P_n P_n(\tau\vert \tau)M_{K,n}^*(s,\tau)H_{n}^*(s)R^{-1}(s)\Vert_1\nonumber\\
&\cdot\Vert H(s)M_{K}(s,\tau)P(\tau\vert \tau)-H_{n}(s)M_{K,n}(s,\tau)P_n(\tau\vert \tau)\mathbf P_n\Vert ds.
\end{align*}
where, according to Lemma \ref{per lemma} and \eqref{approximation covergence of kf},
\begin{eqnarray*}\label{app ineq1}
&&\Vert H(s)M_{K}(s,\tau)P(\tau\vert \tau)-H_{n}(s)M_{K,n}(s,\tau) P_n(\tau\vert \tau)\mathbf P_n\Vert\nonumber\\
&\leqslant&\Vert H(s)M_{K}(s,\tau)\Vert\Vert P(\tau\vert \tau)-P_n(\tau\vert \tau)\mathbf P_n\Vert+\Vert H(s)-H_n(s)\Vert \Vert M_{K}(s,\tau)P_n(\tau\vert \tau)\mathbf P_n\Vert\nonumber\\
&&+ \Vert H_n(s)\Vert \Vert (M_{K}(s,\tau)-M_{K,n}(s,\tau)) P_n(\tau\vert \tau)\mathbf P_n\Vert\rightarrow 0, \quad  n\rightarrow\infty.
\end{eqnarray*}
and so is its adjoint operator.

By the uniform boundedness of $P(t\vert t)$, $M_{K}(t, s)$, $H_n(t)$ for $t\in[t_0,b]$, we have
$
\Vert  P_n(\tau\vert t)\mathbf P_n-P(\tau\vert t)\Vert_1\rightarrow 0, \quad  n\rightarrow\infty.
$
\end{proof}

Now let us take the location of observations into account and show the convergence of optimal observation locations of approximated covariance of Kalman filter and smoother.
\begin{theorem}\label{smoother optimal2}
Assume that the assumptions $(A1)-(A4)$ are held and that $H_{r,n}=H_{r}\mathbf P_n$ and  $\lim_{r\rightarrow r_0}\Vert H_{r}-H_{r_0}\Vert_\infty=0$. If $\mathcal E$ and $\mathcal Y$ are finite-dimensional, $P(t_0\vert t_{-1})$ is nuclear and $\lim_{n\rightarrow \infty}\Vert P_n(t_0\vert t_{-1})\mathbf P_n-P(t_0\vert t_{-1})\Vert_1=0$,
then
\begin{equation*}
\hat{\ell}_{1,n}^f(t)\rightarrow \hat \ell_1^f(t),\quad \hat{r}^f_n\rightarrow \hat r^f, \quad
\hat{\ell}_{1,n}^s(\tau|t)\rightarrow \hat \ell_1^s(\tau|t),\quad \hat{r}^s_n\rightarrow \hat r^s, \quad (t,\tau)\in \Gamma_{t_0}^b, \quad n\rightarrow\infty.
\end{equation*}
\end{theorem}
\begin{proof}
Follows by duality and Theorem \ref{optimal convergence2}.
\end{proof}

\section{Application}\label{lq advection diffusion}
As a popular data assimilation method, the ensemble Kalman filter and smoother are widely applied in meteorology. Hence, we
consider a linear advection-diffusion model  with $\Omega :=(0,5)\times (0,5)\times (0,1)$ on a fixed time interval $[0,3]$ based on the Kalman filter and smoother, the theoretical foundation of the
ensemble Kalman filter and smoother, as an example:
\begin{align*}
&\dfrac{\partial \delta c}{\partial t}=-v_x\dfrac{\partial \delta c}{\partial x}-v_y\dfrac{\partial \delta c}{\partial y}+\dfrac{\partial}{\partial z}(K(z)\dfrac{\partial \delta c}{\partial z})+\delta e-\delta d,\\
&\delta c(t_0)=\delta c_0,\quad \delta e(t_0)=\delta e_0, \quad\delta d(t_0)=\delta d_0,
\end{align*}
where $\delta c$, $\delta e$ and  $\delta d$ are the perturbations of the concentration, the emission rate and deposition rate of a species, respectively.
 $v_x$ and $v_y$ are constants and $K(z)$ is a continuous differentiable function of $z$.

Defining $A_x:=-v_x\frac{\partial}{\partial x}$, $A_y:=-v_y\frac{\partial}{\partial y}$ and $D_z:=\dfrac{\partial}{\partial z}(K(z)\dfrac{\partial} {\partial z})$
with domains
\begin{align*}
 &\mathcal D(A_x)=\{f\in L^2(\Omega) \mid A_xf\in L^2(\Omega),f(0,y,z)=f(5,y,z)\},\\
 &\mathcal D(A_y)=\{f\in L^2(\Omega) \mid A_yf\in L^2(\Omega), f(x,0,z)=f(x,5,z)\},\\
 &\mathcal D(D_z)=\{f\in L^2(\Omega) \mid D_zf\in L^2(\Omega), f_z(x,y,0)=f_z(x,y,1)=0\}
\end{align*}
and denote by $S_{x}$, $S_{y}$ and $S_{z}$  the semigroups generated by $A_x$, $A_y$ and $D_z$. $S$ is the semigroup generated by $A_x+A_y+D_z$ with the domain
$\mathcal D= \mathcal D(A_x)\cap \mathcal D(A_y)\cap \mathcal D(D_z)$.

In particular, in order to include the emission rate into the state vector as optimized parameter, the dynamic model for the emission rate with constant emission factors \cite{Elbern07} is established as
\begin{equation*}
\delta e(t)=M_e(t,s)\delta e(s),
\end{equation*}
where $M_e(t,s)=\frac{e_b(t)}{e_b(s)}\in L(L^2(\Omega))$, $e_b(\cdot)\in L^2(\Omega)$ is termed as the background knowledge of the emission rate, which is
continuous in time and $\sup_{(t,s)\in \Gamma_{0}^{3}}\Vert\frac{e_b(t)}{e_b(s)}\Vert<\infty$.
According to Definition \ref{mild evolution}, it is easy to show that $M_e(\cdot, \cdot)$ is a self-adjoint mild evolution operator.

Ignoring the model error, the model extended with emission rate is given by
\begin{eqnarray}\label{extended model}
&&\left(\begin{array}{c}
      \delta c(t+\triangle t)\\
     \delta e(t+\triangle t)
      \end{array}
\right)\nonumber\\&=&
M(t+\triangle t,t)
\left(\begin{array}{c}
       \delta c(t)\\
     \delta e(t)
      \end{array}
\right)
-\left(\begin{array}{c}
      \int_{t}^{t+\triangle t} S(t+\triangle t- s)\delta d(s)ds\\
     0
      \end{array}
\right)
\end{eqnarray}
where 
\begin{equation*}
M(t+\triangle t,t)=\left(\begin{array}{cc}
      S(\triangle t) & \int_{t}^{t+\triangle t} S(t+\triangle t- s)M_e(s,t)ds\\
      0     & M_e(t+\triangle t,t)
      \end{array}
\right)
\end{equation*}
also satisfies Definition \ref{mild evolution}.

The numerical solution is based on the symmetric operator splitting technique \cite{Batkai12}, \cite{Yanenko71} with space discretization
via finite difference method with discretized intervals $\triangle x$, $\triangle y$ and  $\triangle z$ in three dimensions.
We assume that the grid points $\{r_i\}_{i=1}^n$ have the coordinates $\{(x_{r_i}, y_{r_i}, z_{r_i})\}$  and define the projection $\mathbf P_n: L^2(\Omega)\rightarrow \mathds R^n$
\begin{equation}\label{example projection}
(\mathbf P_nf)_i:=\frac{1}{V_i}\int_{\Omega_i}f(\omega)d\omega, \quad i=1,\cdots, n.
\end{equation}
where $\Omega_i = [x_{r_i}-\frac{\triangle x}{2}, x_{r_i}+\frac{\triangle x}{2}]\times[y_{r_i}-\frac{\triangle y}{2}, y_{r_i}+\frac{\triangle y}{2}]\times[z_{r_i}-\frac{\triangle z}{2}, z_{r_i}+\frac{\triangle z}{2}]$,
$V_i$ is the volume of $\Omega_i$.
Defining $S_n(\triangle t):=S_{x,n}(\frac{\triangle t}{2})S_{y,n}(\frac{\triangle t}{2})S_{z,n}(\triangle t)S_{y,n}(\frac{\triangle t}{2})S_{x,n}(\frac{\triangle t}{2})$, according to \cite[Theorem 3.17]{Batkai09},
we obtain
\begin{equation}\label{splitting convergence}
\lim_{ n\rightarrow \infty, \triangle t\rightarrow 0}\Vert (S_{n}(\triangle t))^{\frac{t}{\triangle t}}\mathbf P_nf-\mathbf P_n S(t)f\Vert=0,\quad f\in L^2(\Omega).
\end{equation}
With the same space discretization for $\delta c$, the approximation of the emission rate is given by
$\mathbf P_n\delta e(t)=M_{e,n}(t,s)\mathbf P_n\delta e(s)$, where $M_{e,n}(t,s)$ is a diagonal matrix with the diagonal given by
\begin{equation*}
\mbox{diag}(M_{e,n}(t,s))=(\frac{\int_{\Omega_1}e_b(t,\omega)d\omega}{\int_{\Omega_1} e_b(s, \omega)d\omega}, \cdots, \frac{\int_{\Omega_n} e_b(t,\omega)d\omega}{\int_{\Omega_n} e_b(s, \omega)d\omega}).
\end{equation*}
Then, we can easily get
\begin{equation*}
\left\Vert \frac{\int_{\Omega_i} e_b(t,\omega)d\omega}{\int_{\Omega_i} e_b(s, \omega)d\omega}(\mathbf P_nf)_i-(\mathbf P_nM_e(t,s)f)_i\right\Vert\rightarrow 0,\quad n\rightarrow\infty,\quad f\in L^2(\Omega),
\end{equation*}
so is the adjoint of $M_e(t,s)$.
The extended model with operator splitting discretized in space can be written as
\begin{eqnarray*}
\left(\begin{array}{c}
      \delta c_n(t+\triangle t)\\
     \delta e_n(t+\triangle t)
      \end{array}
\right)&=&
M_n(t+\triangle t,t)
\left(\begin{array}{c}
       \delta c_n(t)\\
     \delta e_n(t)
      \end{array}
\right)\\
&-&\left(\begin{array}{c}
     S_{x,n}(\frac{\triangle t}{2})S_{y,n}(\frac{\triangle t}{2})\int_{t}^{t+\triangle t} S_{z,n}(t+\triangle t- s)\delta d_n(s)ds\\
     0
      \end{array}
\right),
\end{eqnarray*}
where $\delta c_n(t)=\mathbf P_n\delta c(t)$, $\delta e_n(t)=\mathbf P_n\delta e(t)$ and $\delta d_n(t)=\mathbf P_n\delta d(t)$
\begin{equation*}
M_n(t+\triangle t,t)=\left(\begin{array}{cc}
      S_n(\triangle t) & S_{x,n}(\frac{\triangle t}{2})S_{y,n}(\frac{\triangle t}{2})\int_{t}^{t+\triangle t} S_{z,n}(t+\triangle t- s)M_{e,n}(s,t)ds\\
      0     & M_{e,n}(t+\triangle t,t)
      \end{array}
\right).
\end{equation*}
For any pair of time  $(t,s)\in\Gamma_0^3$, assume $m=\frac{t-s}{\triangle t}\in \mathds N$, we have
\begin{align*}
&\prod_{i=1}^m M_n(s+i\triangle t, s+(i-1)\triangle t)\\&=
\left(\begin{array}{cc}
        (S_n(\triangle t))^m& \sum_{i=1}^m \int_{s+(i-1)\triangle t}^{s+i\triangle t} S_{ce,n}^i(t-h)M_{e,n}(h,s)dh\\
     0& M_{e,n}(t,s)
      \end{array}
\right),
\end{align*}
where $S_{ce,n}^i(t-h)= (S_n(\triangle t))^{m-i}S_{x,n}(\frac{\triangle t}{2})S_{y,n}(\frac{\triangle t}{2})S_{z,n}(s+i\triangle t-h)$, $h\in[s+(i-1)\triangle t,s+i\triangle t]$.
In order to show that $\prod_{i=1}^m M_n(s+i\triangle t, s+(i-1)\triangle t)\mathbf P_n$ is strongly convergent to $\mathbf P_nM(t,s)$, we only need to show
\begin{equation*}
\Vert S_{ce,n}^i(t-h)M_{e,n}(h,s)\mathbf P_nf-\mathbf P_nS(t-h)M_e(h,s)f\Vert\rightarrow 0, \quad m,n\rightarrow \infty.
\end{equation*}
In fact,
\begin{align*}
&\Vert S_{ce,n}^i(t-h)M_{e,n}(h,s)\mathbf P_nf-\mathbf P_nS(t-h)M_e(h,s)f\Vert\\
&\leqslant\Vert S_{ce,n}^i(t-h)M_{e,n}(h,s)\mathbf P_nf-S_{ce,n}^i(t-h)\mathbf P_nM_e(h,s)f\Vert\\
&+\Vert S_{ce,n}^i(t-h)\mathbf P_nM_e(h,s)f-\mathbf P_nS(t-h)M_e(h,s)f\Vert,
\end{align*}
where, clearly, $\Vert S_{ce,n}^i(t-h)M_{e,n}(h,s)\mathbf P_nf-S_{ce,n}^i(t-h)\mathbf P_nM_e(h,s)f\Vert\rightarrow 0,\quad m,n\rightarrow \infty$.
Moreover, we have
\begin{align*}
&\Vert S_{ce,n}^i(t-h)\mathbf P_nM_e(h,s)f-\mathbf P_nS(t-h)M_e(h,s)f\Vert\\
\leqslant& \Vert ((S_n(\triangle t))^{m-i}-S(t-s-i\triangle t))S_{x,n}(\frac{\triangle t}{2})S_{y,n}(\frac{\triangle t}{2})S_{z,n}(s+i\triangle t-h)\mathbf P_nM_e(h,s)f\Vert\\
+&\Vert S(t-s-i\triangle t)( S_{x,n}(\frac{\triangle t}{2})S_{y,n}(\frac{\triangle t}{2})S_{z,n}(s+i\triangle t-h)\mathbf P_n\\&- \mathbf P_nS(s+i\triangle t-h))M_e(h,s)f \Vert,
\end{align*}
where, according to \eqref{splitting convergence}, $\Vert ((S_n(\triangle t))^{m-i}-\mathbf P_nS(t-s-i\triangle t))S_{x,n}(\frac{\triangle t}{2})S_{y,n}(\frac{\triangle t}{2})S_{z,n}(s+i\triangle t-h)\mathbf P_nM_e(h,s)f\Vert\rightarrow 0$ and
\begin{eqnarray*}
&&\Vert  (S_{x,n}(\frac{\triangle t}{2})S_{y,n}(\frac{\triangle t}{2})S_{z,n}(s+i\triangle t-h)\mathbf P_n- \mathbf P_nS(s+i\triangle t-h))M_e(h,s)f \Vert\\
&\leqslant& \Vert S_{x,n}(\frac{\triangle t}{2})S_{y,n}(\frac{\triangle t}{2})S_{z,n}(s+i\triangle t-h)\Vert\\
&&\cdot\Vert(I- S_{z,n}(h-s-(i-1)\triangle t))\mathbf P_nM_e(h,s)f \Vert\\
&+&\Vert (S_{x,n}(\frac{\triangle t}{2})S_{y,n}(\frac{\triangle t}{2})S_{z,n}(\triangle t)-S_n(\triangle t))\mathbf P_nM_e(h,s)f \Vert\\
&+& \Vert (S_n(\triangle t)\mathbf P_n-\mathbf P_nS(\triangle t))M_e(h,s)f \Vert\\
&+&\Vert (S(h-s-(i-1)\triangle t)-I)S(s+i\triangle t-h)M_e(h,s)f\Vert\rightarrow 0, \ \triangle t\rightarrow 0,\ n\rightarrow \infty.
\end{eqnarray*}

Further, we discretize the model in time by the Lax-Wendroff scheme for advection equations in horizontal directions and Crank-Nicolson scheme for the diffusion equation
in the vertical direction such that $S_{x/y/z,n}$ are approximated by
\begin{align*}
&\tilde S_{x/y,n}(\frac{\triangle t}{2})=I+\frac{\triangle t}{2}A_{x/y,n}+\frac{\triangle t^2}{8}A_{x/y,n}^2, \\
&\tilde S_{z,n}(\triangle t)=(I-\frac{\triangle t}{2}D_{z,n})^{-1}(I+\frac{\triangle t}{2}D_{z,n}),\\
&\tilde B_{z,n}^e(t,s)f=(I-\frac{\triangle t}{2}D_{z,n})^{-1}(\frac{\triangle t}{2}M_{e,n}(t,s)f),
\end{align*}
where $A_{x/y,n}$ and $D_{z,n}$ is the approximate generators to $n$-dimensional state space based on finite difference methods.

It is well known \cite{Dautray99} that the Lax-Wendroff scheme is consistent and conditional stable for $A_x$ and $A_y$ and the Crank-Nicolson scheme is consistent and stable for $D_z$, $(I-\frac{\triangle t}{2}D_{z,n})^{-1}$ is the consistent and condition stable implicit Euler scheme,
by Lax equivalence theorem, that is
\begin{align*}
\lim_{\triangle t\rightarrow 0}\Vert(\tilde S_{x/y/z,n}(\triangle t))^{\frac{t}{\triangle t}}f-S_{x/y/z,n}(t)f\Vert=0, \quad f\in L^2(\Omega),\\
\lim_{\triangle t\rightarrow 0}\Vert((I-\frac{\triangle t}{2}D_{z,n})^{-1})^{\frac{2t}{\triangle t}}f-S_{z,n}(t)f\Vert=0, \quad f\in L^2(\Omega).
\end{align*}
Similarly defining $\tilde S_n:=\tilde S_{x,n}\tilde S_{y,n}\tilde S_{z,n}\tilde S_{y,n}\tilde S_{x,n}$,
\begin{equation}\label{convergence discrete time}
\lim_{ n\rightarrow \infty, \triangle t\rightarrow 0}\Vert (\tilde S_n(\triangle t))^{\frac{t}{\triangle t}}\mathbf P_nf-\mathbf P_nS(t)f\Vert=0,\quad f\in L^2(\Omega).
\end{equation}
Since $A_x$, $A_y$ and $D_z$ are self-adjoint, $\tilde S_n^*(\triangle t))^{\frac{t}{\triangle t}}$ is also strongly convergent to $S^*(t)$.

Thus, \eqref{extended model} is approximated by
\begin{eqnarray*}
&&\left(\begin{array}{c}
      \delta \tilde c_n(t+\triangle t)\\
     \delta \tilde e_n(t+\triangle t)
      \end{array}
\right)\\&=&
\left(\begin{array}{cc}
      \tilde S_{n}(\triangle t)& \tilde S_{x,n}(\frac{\triangle t}{2})\tilde S_{y,n}(\frac{\triangle t}{2})\tilde B_{z,n}^e(t+\triangle t,t)\\
      0     & M_{e,n}(t+\triangle t,t)
      \end{array}
\right)
\left(\begin{array}{c}
       \delta \tilde c_n(t)\\
     \delta \tilde e_n(t)
      \end{array}
\right)\\
&-&
\left(\begin{array}{c}
     \tilde S_{x,n}(\frac{\triangle t}{2})\tilde S_{y,n}(\frac{\triangle t}{2})(I-\frac{\triangle t}{2}D_{z,n})^{-1}[\frac{\triangle t}{2}(\delta d_n(t+\triangle t)+\delta d_n(t))]\\
     0
      \end{array}
\right).
\end{eqnarray*}
Defining the above block evolution operator as $\tilde M_n(t,s)$, $(t,s)\in \Gamma_0^3$, we have
\begin{eqnarray*}
&&\prod_{i=1}^m \tilde M_n(s+i\triangle t, s+(i-1)\triangle t)\\&=&
\left(\begin{array}{cc}
        (\tilde S_n(\triangle t))^m& \sum_{i=1}^m(\tilde S_n(\triangle t))^{m-i}\tilde S_{x,n}(\frac{\triangle t}{2})\tilde S_{y,n}(\frac{\triangle t}{2})\tilde B_{z,n}^e(s+i\triangle t,s)\\
     0& M_{e,n}(t,s)
      \end{array}
\right).
\end{eqnarray*}
We define by 
\begin{eqnarray*}
&&B_{z,n}^e(s+i\triangle t,s+(i-1)\triangle t, s)f\\&:=&\int_{s+(i-1)\triangle t}^{s+i\triangle t}S_{z,n}(s+i\triangle t-h)M_{e,n}(h,s)fdh, \quad f\in L^2(\Omega).
\end{eqnarray*}
By the trapezoidal rule and convergence of the implicit Euler scheme, we have
\begin{eqnarray*}
&&\Vert \tilde B_{z,n}^e(s+i\triangle t,s)f-B_{z,n}^e(s+i\triangle t,s)f\Vert\\
&\leqslant&\Vert((I-\frac{\triangle t}{2}D_{z,n})^{-1}-S_{z,n}(\triangle t))(\frac{\triangle t}{2}M_{e,n}(s+i\triangle t,s)f)\Vert\\
&+&\Vert\frac{\triangle t}{2}M_{e,n}(s+i\triangle t,s)f \Vert+\Vert \frac{\triangle t}{2}S_{z,n}(\triangle t)M_{e,n}(s+(i-1)\triangle t,s)f\\&&+\frac{\triangle t}{2}M_{e,n}(s+i\triangle t,s)f-B_{z,n}^e(s+i\triangle t,s)f\Vert \\
&+&\Vert \frac{\triangle t}{2}S_{z,n}(\triangle t)(M_{e,n}(s+i\triangle t,s)f-M_{e,n}(s+(i-1)\triangle t,s)f)\Vert\rightarrow 0, \ \triangle t\rightarrow 0.
\end{eqnarray*}

For the observation system, we assume there is only a single observation during the entire time interval and define the observation mapping
$H_r: L^2(\Omega)\rightarrow \mathds R$ by
$$
H_rf:=\frac{1}{ V_r}\int_{\Omega_r}f(\omega)d\omega, \quad r=(x_r,y_r,z_r), \quad f\in L^2(\Omega),
$$
where $\Omega_r$ and $V_r$ are similarly defined as \eqref{example projection}.
Then, the observation system extended with the emission rate is given by
\begin{equation*}
 \delta y(t)=(H_r,0)
\left(\begin{array}{c}
       \delta c(t)\\
     \delta e(t)
      \end{array}
\right)+\nu(t),
\end{equation*}
where $\delta y(t)\in \mathds R$ and $\nu(t)$ is the white noise with distribution $N(0,1)$.

According to the spatial discretization of the model, in the vertical direction, $[0,1]$ is discretized into three layers $\{0,0.5,1\}$.
Since the diffusion coefficient $K(z)$ is small, we assume possible locations of the single observation are around the grid points in the first layer $z=0$.

We have already shown that the assumptions $(A1)-(A3)$ in Section \ref{opt loc of kf and ks} and the compactness of the possible area of observation locations are satisfied.

In addition, according to the spatial discretization, we assume that the initial covariance is given by $P_n(t_{0}\vert t_{-1})=e^{-8}I_n$, where $I_n$ is the $n\times n$ identity matrix.
It implies that $P_n(t_0|t_{-1})$ does not converge to a nuclear operator.
It is shown in Figure 1 that the optimal location and minimal cost based on Kalman filter do not converge in this situation.
\begin{figure*}
\centering
\includegraphics[scale=0.7]{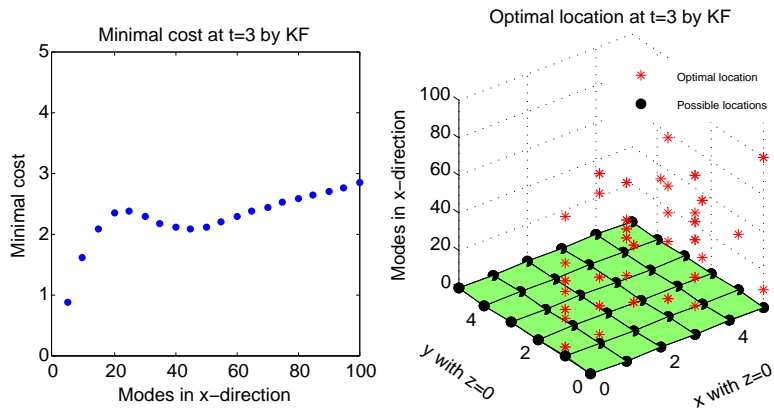}
\caption{\footnotesize Minimal cost and optimal location based on Kalman filter without the nuclearity of $P(t_0|t_{-1})$.}
\end{figure*}
\begin{figure*}
\centering
\includegraphics[scale=0.7]{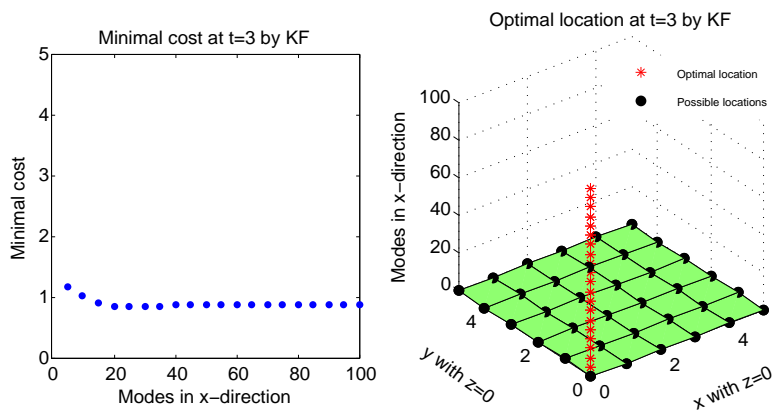}
\caption{\footnotesize Minimal cost and optimal location of the estimation of the state at final time by Kalman filter.}
\end{figure*}
\begin{figure*}
\centering
\includegraphics[scale=0.7]{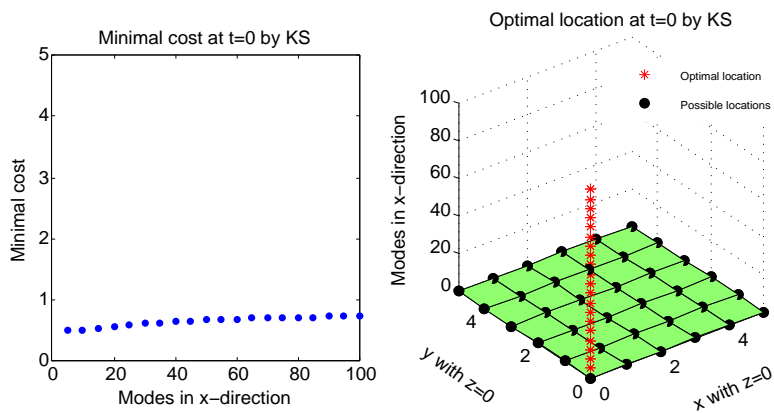}
\caption{\footnotesize Minimal cost and optimal location of the estimation of the initial state by Kalman smoother.}
\end{figure*}

Then, according to \eqref{convergence discrete time} and dominated convergence theorem, we obtain that
$$\sum_{i=1}^m(\tilde S_n(\triangle t))^{m-i}\tilde S_{x,n}(\frac{\triangle t}{2})\tilde S_{y,n}(\frac{\triangle t}{2})\tilde B_{z,n}^e(s+i\triangle t,s)$$ is strongly convergent to
$$\sum_{i=1}^m (S_n(\triangle t))^{m-i}S_{x,n}(\frac{\triangle t}{2})S_{y,n}(\frac{\triangle t}{2})B_{z,n}^e(s+i\triangle t,s).$$ Further, 
$\prod_{i=1}^m \tilde M_n(s+i\triangle t, s+(i-1)\triangle t)$ is strongly convergent to $\prod_{i=1}^m M_n(s+i\triangle t, s+(i-1)\triangle t)$.

Next we define the initial covariance as
\begin{equation*}
P(t_{0}\vert t_{-1})f=\sum_{i=1}^{\infty}e^{-i^2}\langle f,e_i\rangle e_i,\quad f\in L^2(\Omega),
\end{equation*}
where $\{e_i\}$ is an orthogonal basis of $L^2(\Omega)$. The $n$-dimensional approximation of $P(t_{0}\vert t_{-1})$ is given by
\begin{equation*}
P_n(t_{0}\vert t_{-1})\mathbf P_nf=\sum_{i=1}^{n}e^{-i^2}\langle\mathbf P_n f,e_i\rangle e_i,\quad f\in L^2(\Omega).
\end{equation*}
With this choice, $P(t_0|t_{-1})$ is nuclear and the assumption $(A4)$ in Section \ref{opt loc of kf and ks} is satisfied.
By Theorem \ref{smoother optimal2},
the optimal location and minimal cost based on Kalman filter and smoother are convergent, which are shown in Figure 2 and Figure 3, respectively.

\end{document}